\documentclass{amsart}

\usepackage[english]{babel}
\usepackage[utf8]{inputenc}
\usepackage{csquotes}
\usepackage{bbm}
\usepackage{hyperref}
\usepackage{verbatim}
\usepackage{mathabx}
\usepackage{tabularx,ragged2e}
\newcolumntype{C}{>{\Centering\arraybackslash}X}

\usepackage[style=numeric,maxcitenames=2,maxbibnames=10,doi=false,isbn=false,url=false,natbib=true,hyperref,bibencoding=utf8,backend=biber]{biblatex}
\AtEveryBibitem{%
  \clearfield{note}%
  \clearfield{issn}%
  \clearlist{language}%
  }

\usepackage{amssymb,graphicx}
\usepackage{enumitem}
\usepackage[table]{xcolor}
\usepackage{booktabs}

\newcommand{\N}{\ensuremath{\mathbb{N}}}
\newcommand{\R}{\ensuremath{\mathbb{R}}}

\newcommand{\Z}{\ensuremath{\mathbb{Z}}}
\newcommand{\E}{\ensuremath{\mathbb{E}}}
\renewcommand{\P}{\ensuremath{\mathbb{P}}}

\newcommand{\ind}[1]{\ensuremath{\mathbbm{1}_{\{#1\}}}}
\newcommand{\diff}{\mathop{}\mathopen{}\mathrm{d}}
\newcommand{\cal}[1]{\ensuremath{\mathcal{#1}}}

\newcommand\steq[1]{\stackrel{\text{\rm #1.}}{=}}

\def\eps{\varepsilon}
\def\cadlag{c\`adl\`ag }

\usepackage{enumitem}

\newtheorem{proposition}{Proposition}
\newtheorem{definition}[proposition]{Definition}
\newtheorem{lemma}[proposition]{Lemma}
\newtheorem{theorem}{Theorem}
\newtheorem{corollary}[proposition]{Corollary}

\title[]{A Palm Space Approach\\ to Non-Linear Hawkes Processes}
\date{\today}
\author[Ph. Robert]{Philippe Robert}
\email{Philippe.Robert@inria.fr}
\urladdr{http://www-rocq.inria.fr/who/Philippe.Robert}
\address[Ph. Robert, G. Vignoud]{INRIA Paris, 2 rue Simone Iff, 75589 Paris Cedex 12, France}
\author[G. Vignoud]{Ga\"etan Vignoud}
\email{Gaetan.Vignoud@inria.fr}
\address[G. Vignoud]{ Center for Interdisciplinary Research in Biology (CIRB) - Collège de France (CNRS UMR 7241, INSERM U1050), 11 Place Marcelin Berthelot, 75005 Paris, France}
\thanks{${}^1$Supported by PhD grant of \'Ecole Normale Sup\'erieure, ENS-PSL}


\date{\today}

\begin{document}

\begin{abstract}
A Hawkes process on $\R$ is a point process whose intensity function at time $t$ is a functional of its past activity before time $t$. It is defined by its activation function $\Phi$ and its memory function $h$. In this paper, the Hawkes property is expressed as an operator on the sub-space of non-negative sequences associated to distances between its points.  By using the classical correspondence between a  stationary point process and its Palm measure, we establish a  characterization of the corresponding Palm measure as an invariant distribution of a Markovian kernel.  We prove that if $\Phi$ is continuous and its growth rate is at most linear with a rate below some constant, then there exists a stationary Hawkes point process. The classical Lipschitz condition of the literature for an unbounded function $\Phi$  is relaxed. Our proofs rely on a  combination of coupling methods, monotonicity properties of linear Hawkes processes  and classical results on Palm distributions. An investigation of the Hawkes process starting from the null measure, the empty state,  on $\R_-$ plays also an important role. The linear case of Hawkes and Oakes  is revisited at this occasion.

If the memory function $h$ is an  exponential function, under a weak condition it is shown that there exists a unique stationary Hawkes point process. In this case,  its Palm measure is expressed in terms of the invariant distribution of a one-dimensional Harris ergodic Markov chain.  When the activation function is a polynomial $\Phi$ with degree ${>}1$,  there does not exist  a stationary Hawkes process and if the Hawkes process starts from the empty state, a scaling result for the accumulation of its points is  obtained. 

\end{abstract}

\maketitle

 \vspace{-5mm}

\bigskip

\hrule

\vspace{-3mm}

\tableofcontents

\vspace{-1cm}

\hrule

\bigskip

\section{Introduction}
A point process ${\cal N}{=}(t_n:n{\ge}0)$ on $\R_+$ has the Hawkes property if conditionally on  the  points $t_n$, $n{\ge}1$ of 
 ${\cal N}$ up to time $t$, the rate at which a new jump occurs in the time interval $(t,t{+}\diff t)$ is given
by
\[
\Phi\left(\sum_{n: t_n{\le}t} h(t{-}t_n)\right)\diff t= \Phi\left(\int_{-\infty}^t h(t{-}u){\cal N}(\diff u)\right)\diff t,
\]
This is described as a self-excitation property of the dynamic. It can be formulated as the fact that the  process 
\[
\left({\cal N}(0,t){-}\int_0^t  \Phi\left(\int_{-\infty}^s h(s{-}u){\cal N}(\diff u)\right)\diff s \right)
\]
is a local martingale with respect to a convenient filtration.

\subsection*{The Parameters of the Hawkes Process}
\begin{enumerate}
\item The function $h$, the {\em  memory function}.  The quantity $h(t{-}s)$ gives the residual impact at time $t$ of a jump which has occurred at time $s{\le}t$.  It is assumed that it is a non-increasing continuous function, such that the function $(th(t))$ is converging to $0$ at infinity and 
\begin{equation}\label{alpha}
\alpha\steq{def}\int_0^{+\infty} h(u)\diff u <{+}\infty.
\end{equation}
\item The {\em activation function} $\Phi$ modulates the global impact of past jumps. It is assumed that it is continuous and 
\begin{equation}\label{beta}
\beta\steq{def} \limsup_{t\to+\infty} \frac{\Phi(t)}{t}< {+}\infty. 
\end{equation}
\end{enumerate}

This class of models has been used in numerous situations such as, mathematical finance, \citet{bauwens_modelling_2009},
population dynamics, \citet{boumezoued_population_2016}, biology, \citet{reynaudbouret:hal-00863958}, queueing systems,
\citet{daw_queues_2017}, learning theory, \citet{etesami_learning_2016}, or neurosciences, \citet{gerhard_stability_2017}, \ldots

This work is initially motivated by our
investigations of mathematical models of plastic synapticity in neural networks, see~\citet{robert_stochastic_2020}, \cite{robert_stochastic_2020_1} \cite{robert_stochastic_2020_2}.

From a mathematical point of view, these processes have generated, and still generate, a considerable interest. Pioneering works on the existence and uniqueness of  stationary point process, i.e. when the distribution of the point process is invariant with respect to translation, are:
\begin{itemize}
\item \citet{hawkes_cluster_1974}  shows that when the activation function is affine, these processes can be
represented by an age-dependent branching process, a special class of the so-called Crump-Mode-Jagers models. The branching property is an important feature in the analysis of these models. 
\item \citet{kerstan_teilprozesse_1964} gives an existence and uniqueness result of  a stationary point process  in the more general context of a fixed point relation for the distribution of random measures. A contraction argument in a convenient functional setting is the key ingredient. The reference \citet{bremaud_stability_1996} has considerably developed this method in the case of Hawkes processes.

In this setting, the main existence and uniqueness result for an unbounded activation function $\Phi$ and a general memory function $h$ is obtained under the condition that $\Phi$ is Lipschitz with coefficient $L_\Phi$ and that the relation $\alpha L_\Phi{<}1$ holds, where $\alpha$ is defined by Relation~\eqref{alpha}.  Note that $L_{\Phi}{\ge}\beta$, with $\beta$ defined by Relation~\eqref{beta}. We will see in this paper that the weaker conditions that $\Phi$ is continuous and that $\alpha\beta{<}1$ are enough for the existence result.
\end{itemize}
Up to now, these are mainly the two  main approaches to investigate the existence of stationary  Hawkes processes. Renewal properties are also an important tool in the study of Hawkes processes. See~\citet{graham_regenerative_2019, raad_renewal_2019, costa_renewal_2020}.  Appendix~\ref{App-Rev} is an attempt to present  some aspects of the overwhelming literature of this domain in a table. Functional limit theorems like law of large numbers or central limit theorems and  large deviations properties are also topics of interest for these processes. 

In our framework, assuming that a point process is given on $\R_-$, the Hawkes property  on $\R_+$ is expressed with a solution of a Stochastic Differential Equation (SDE) satisfied by its counting process. We reformulate this property  in terms of the distance between the  jump times of the point process.  In the case of a stationary Hawkes point process, this gives a Markovian characterization of  its Palm measure in terms of a positive recurrence property.

The natural state space to study Palm measures is a sub-space of non-negative sequences, of distances between the successive points of the point process. Unfortunately the appropriate state space is not complete as a metric space and, moreover,  the Markov process does not have the Feller property so that the classical tools to prove positive recurrence do not seem to be available. See~\citet{hairer_convergence_2010} for a quick presentation. Nevertheless, our main result, Theorem~\ref{ExistenceTheo} shows that under  appropriate, weak, conditions on $\Phi$ there exists such a Palm measure and therefore a stationary Hawkes process.  It should be noted however that we have not been able to obtain a significant uniqueness result under these weak conditions. Contraction arguments, which are apparently not possible under our weak assumptions,  may be required for that purpose.  Nevertheless, an interesting coupling property, Corollary~\ref{corlc1},  is proved. 

Our approach uses the Markov chain starting from the empty state, i.e. the solution of the SDE when the initial state is the Dirac measure at $0$ (no activity on $\R_-$ except at $0$). Proposition~\ref{TnIneq} shows  that the points of the corresponding point process are lower-bounded by a functional of a simple random walk. A simple characterization of non-homogeneous point processes, see Proposition~\ref{HSDEprop}, plays an important role to derive this result.  The existing literature of Hawkes processes relies more on a stochastic calculus approach, via a formulation in terms of previsible projections of stochastic intensities. 

A second ingredient is a coupling with an linear Hawkes process, the Hawkes process of~\citet{hawkes_cluster_1974} when $\Phi$ is an affine function.   Our analysis of linear Hawkes processes does not explicitly use  the natural branching property of this model. Monotonicity properties and technical estimates give then the desired existence result under weak conditions. See the discussions of Section~\ref{App-Rev} of the appendix and at the beginning of Section~\ref{ExisSec}. 

When the memory function is exponential, an existence result of a stationary Hawkes process  is obtained under an even weaker condition and an explicit representation of the Palm measure is obtained. We also analyze the transient case for which there are few studies in general. In this case, the non-decreasing sequence of points of the Hawkes dynamics blows-up, i.e. converges almost surely to a finite limit. The self-excitation property of the dynamic  leads to an explosion in finite time.  We derive a scaling result, see Theorem~\ref{TransientH},  for the sequence of point in terms of a Poisson process.  

The paper is organized as follows.  Section~\ref{PPSubSec} introduces some basic definitions, Section~\ref{HSDEsec} expresses the Hawkes property in terms of a stochastic differential equation and Section~\ref{MarkSec} gives a characterization of the Palm measure of a stationary Hawkes process as an invariant measure of a Markov chain in the state space of non-negative sequences.  Section~\ref{SecCrump} revisits a classical Hawkes process, when the activation function is affine.   Section~\ref{ExisSec} gives the main existence of this paper with this approach and by using coupling techniques and some monotonicity properties.  Section~\ref{ExpDecSec} investigates Hawkes process with an exponential memory function. Appendix~\ref{AppPP} presents the main definitions and results concerning stationary point processes. Appendix~\ref{App-Rev} gives a quick review of the results and the methods for the existence and uniqueness of stationary Hawkes processes.

\section{Definitions and Notations}\label{PPSubSec}
\subsection{Probability space}

It is assumed that on the  probability space $(\Omega,{\cal F}, \P)$ is defined a Poisson point process ${\cal P}$ on $\R_+{\times}\R$ with intensity measure $\diff x{\otimes}\diff y$. See~\citet{kingman_poisson_1992,last_lectures_2017} and Chapter~1 of~\citet{robert_stochastic_2003}  for a brief account on Poisson processes.

Additionally  $({\cal F}_t)$ is a filtration such that, for $t{\in}\R$,  ${\cal F}_t$ is
the $\sigma$-field generated by the random variables ${\cal P}(A{\times}(s,t])$, where $A$ is a Borelian set of $\R_+$ and
$s{\le}t$. If $\lambda$ and $f$ are non-negative Borelian functions $f$ on $\R$, we define
  \[
  \int_{\R} f(s){\cal P}((0,\lambda(s)],\diff s) \steq{def}
    \int_{\R_+{\times}\R} f(s)\ind{v{\le}\lambda(s)} {\cal P}(\diff v, \diff s),
    \]
    and it is assumed that the $\sigma$-field ${\cal F}$ of the probability space verifies
    \[
    \bigcup_{t{\ge}0} {\cal F}_t \subset {\cal F}. 
    \]
The martingale and stopping time properties are understood with respect to this filtration.

If $H$ is $\R$ or $\R_-$, we denote by $C_c(H)$ the space of continuous functions on $H$ with compact support and
 ${\cal M}_p(H)$, the set of Radon point measures on $H$, that is positive Radon measures carried by points, for
 $m{\in}{\cal M}_p(H)$ then
\[
m{=}\sum_{x{\in}S} \delta_{x},
\]
where $\delta_x$ is the Dirac measure at $x{\in}H$ and $S$ is a countable subset of $H$ with no limiting points in $H$.
We may also  represent $m$ as a sequence $(x_n,n{\in}\N)$ of points. 
If $A$ is a subset of $H$, we denote by 
\[
m(A)=\int_A m(\diff x) =\sum_{x{\in}S}\ind{x{\in}A},
\]
the number of points of $m$ in $A$. A point measure $m$ is {\em simple} if $m(\{x\}){\in}\{0,1\}$, for all $x{\in}H$.
The space ${\cal M}_p(H)$ is endowed with the topology of weak convergence.

\subsection{State space of non-negative sequences}
We refer to Appendix~\ref{AppPP} for general definitions concerning point processes.
We denote by ${\cal S}$ the sub-space of sequences of non-negative real numbers
\begin{equation}\label{eqS}
{\cal S}\steq{def}\left\{x{=}(x_k){\in}(\R_+{\cup}\{{+}\infty\})^{\N\setminus\{0\}}:
x_{k_0}{=}{+}\infty {\Rightarrow} x_{k}{=}{+}\infty,\, \forall k{\ge}k_0 \right\}.
\end{equation}

\subsection*{Correspondence between ${\cal S}$ and ${\cal M}_p(\R_-)$} \ \\
A functional  from ${\cal S}$ to the space of positive measures with a mass at $0$ is introduced as
follows, if $x{=}(x_k){\in}{\cal S}$, the positive measure $m_x$ on $\R_-$ is defined by
\begin{equation}\label{mx}
m_x=\delta_0+\sum_{k=1}^{+\infty} \delta_{t_{k}},\quad \text{ with } t_{k}{=}{-}\!\!\sum_{i=1}^{k}x_i,\, k{\ge}1,
\end{equation}
with the convention that  $\delta_{-\infty}{\equiv}0$.

The measure $m_x$ is a positive measure carried by points associated to $x{\in}{\cal S}$.
Note that $m_x$ is not necessarily a point measure, i.e it may not have Radon property, since we do
not exclude the fact that the sequence $(x_k)$ converges to $0$ sufficiently fast so that the measure $m_x$ may have
a finite limiting point.

With this definition $0$ is always a point of $m_x$ and that the coordinates of $x$ are
the inter-arrivals times of $m_x$, in particular $x_1$ is the distance to the first point of $m_x$ on the left of $0$.
The point measures with a finite number of points correspond to sequences $(x_k)$ which are constant and equal to
${+}\infty$ after some finite index.
In this case, if $k_0$ is the first index where $x_{k_0}{=}+\infty$, with a slight abuse of notation we will write it as a finite
vector $x{=}(x_1,x_2,\ldots, x_{k_0-1},{+}\infty)$ or $x{=}(x_1,x_2,\ldots, x_{k_0-1})$.

On $ {\cal S}$,  the distance, for $x{=}(x_k)$, $y{=}(y_k){\in}{\cal S}$,
\begin{equation}\label{DistS}
d(x,y)=\sum^{+\infty}_{1} \frac{1}{2^k}\,\min(|x_{k}{-}y_{k}|,1),
\end{equation}
with the convention, for $u{\in}\R_+$, $|u{-}\infty|{=}|\infty{-}u|{=}{+}\infty$ and $|\infty{-}\infty|{=}0$.

An important subset of ${\cal S}$ is
\begin{equation}\label{eqS+}
{\cal S}_h\steq{def}\left\{x{=}(x_k){\in}{\cal S}: h(0) + \sum_{k=1}^{+\infty} h\left(\sum_{i=1}^{k}x_{i}\right) {=}
\int_0^{+\infty}h({-}u)\,{m}_x(\diff u) {<}{+}\infty\right\},
\end{equation}
we have $x{\in}{\cal S}_h$ if and only if  $(h({-}u)){\in}L_1({m}_x)$.
Throughout the paper, we will use the following convention, for $t{\ge}0$ and $x{\in}{\cal S}_h$,
\begin{equation}\label{Conv}
\sum_{k=0}^{+\infty} h\left(t{+}\sum_{i=1}^{k}x_{i}\right)= h(t) + \sum_{k=1}^{+\infty} h\left(t{+}\sum_{i=1}^{k}x_{i}\right).
\end{equation}
\begin{definition}\label{Tdef}
 If $\Phi(0){>}0$,  for $a{\ge}0$ and $x{\in}{\cal S}_h$, we set 
\begin{align}
{\cal T}(x,a)&\steq{def}\inf_{t{\ge}0}\left\{\int_{0}^{t}\Phi\left(h(s) + \sum_{k{\ge}1}
          h\left(s{+}\sum_{i=1}^{k}x_i\right)\right)\,\diff s\ge a\right\}\label{Tau}\\
  &=  \inf_
  {t{\ge}0}\left\{\int_{0}^{t}\Phi\left(\int_{(\infty,s)} h\left(s{-}u\right)\,m_x(\diff u)\right)\diff s\ge a\right\},\notag
\end{align}
 and ${\cal T}(x,a){\steq{def}}0$ otherwise, i.e. if $x{\in}{\cal S}{\setminus}{\cal S}_h$, with the convention $h({+}\infty){=}0$.
\end{definition}

\begin{lemma}\label{Lem1}
  If $\Phi(0){>}0$ and $x{\in}{\cal S}_h$   then,
  \[
  \lim_{t\to +\infty} \int_{0}^{t}\Phi\left(\int_{({-}\infty,s)} h\left(s{-}u\right)\,m_x(\diff u)\right)\diff s={+}\infty.
  \]
In particular,  for $x{\in}{\cal S}_h$ and $a{\ge}0$, the variable ${\cal T}(x,a)$ is finite. 
\end{lemma}

\begin{proof}
Note that, for any $t{\ge}0$,  the monotonicity property of $h$ gives $h(t{-}u){\le}h(-u)$ and, since $x{\in}{\cal S}_h$,
  \[
\int_{(-\infty,0)}h({-}u)m_x(\diff u) <{+}\infty,
  \]
with Lebesgue's dominated convergence Theorem  we obtain the identity
\begin{equation}\label{eqH1}
\lim_{t\to+\infty} \int_{(-\infty,0)}h(t{-}u)m_x(\diff x)=0.
\end{equation}
Our lemma is proved since $\Phi$ is continuous with $\Phi(0){>}0$.
\end{proof}

\section{Hawkes SDEs}
\label{HSDEsec}

We first recall the classical definition of an Hawkes process, see~\citet{hawkes_cluster_1974}.
\begin{definition}
A {\em Hawkes process } is a point process ${\cal N}$ on $\R$ such that, for any $s{\in}\R$, the process 
\[
\left({\cal N}((s,t]){-}\int_s^t \Phi\left(\int_{(-\infty,u)} h(u{-}x){\cal N}(\diff x)\right)\diff u, t{\ge}s\right)
  \]
  is a local martingale with respect to the filtration $({\cal F}_t,t{\ge}s)$, where, for $t{\in}\R$,  ${\cal F}_t$ is
  the $\sigma$-field containing   the $\sigma$-field associated to the random variables ${\cal N}((u,v])$, $u{\le}v{\le}t$.
\end{definition}
If ${\cal N}$ is a Hawkes process, for any $s{\in}\R$  the dual predictable projection of the process of
$({\cal N}((s,t]), t{\ge}s)$ is almost surely,
  \[
\left(\int_s^t \Phi\left(\int_{(-\infty,u)} h(u{-}x){\cal N}(\diff x)\right)\diff u, t{\ge}s\right),
\]
see Theorem~VI~(21.7) of~\citet{rogers_diffusions_2000} for example.
In the terminology of random measures,  see~\citet{Jacod}, the {\em stochastic intensity} of  ${\cal N}$ is
  \[
\left( \Phi\left(\int_{(-\infty,u)} h(u{-}x){\cal N}(\diff x)\right), u{\in}\R\right).
  \]
We now introduce a dynamical system extending  a Radon measure on $\R_-$ into a measure on $\R$ exhibiting a Hawkes property on $\R_+$. The Markovian approach used in this paper relies heavily on this construction.

\begin{proposition}[Hawkes SDE]
\label{HSDEprop}
If  $m{\in}{\cal M}_p(\R_-)$ is such that $(h({-}u)){\in}L_1({m})$, then there exists a unique positive random measure ${\cal N}_m$ on $\R$ such that ${\cal N}_m{\equiv}m$ on $\R_-$ and the counting measure $({\cal N}_m((0,t]),t{\ge}0)$ satisfies the stochastic differential equation 
\begin{equation}\label{HSDE}
  \diff {\cal N}_m((0,t])={\cal P}\left(\left(0,\Phi\left(\int_{(-\infty,t)} h(t{-}x){\cal N}_m(\diff x)\right)\right),
  \diff t \right),
\end{equation}
for all $t{>}0$.

If $\Phi(0){>}0$, the points of ${\cal N}_m$ on $\R_+$ is a non-decreasing sequence of stopping times $(T_n,n{\ge}1)$, such that if
\begin{equation}\label{eqX}
E_{n}\steq{def} \int_{T_{n-1}}^{T_{n}}\Phi\left(\int_{({-}\infty,s)}
h\left(s{-}x\right){\cal N}_m(\diff x)\right)\,\diff s, \quad n{\ge}1,
\end{equation}
with the convention $T_0{=}0$, then $(E_n, n{\ge}1)$ is an i.i.d. sequence of exponential random variables with parameter $1$ and, for $n{\ge}1$,
the sequence $(E_{k},k{>}n)$ is independent of ${\cal F}_{T_{n}}$.
\end{proposition}

\begin{proof}
We construct by induction the sequence of points $(T_n)$ of ${\cal N}_m$.
The first point is defined as 
\[
T_1\steq{def} \inf\left\{t{\ge}0: \int_{(0,t]}{\cal P}
\left( \left(0,\Phi\left(\int_{(-\infty,0]} h(s{-}x)m(\diff x)\right)\right),\diff s\right){\ne}0\right\}.
\]
Let ${\cal R}$ be a Poisson process on $\R_+$ with intensity $\diff x$, the relation
\begin{multline}\label{eqpp}
\left(\int_{(0,t]}{\cal P}\left(\left(0,\Phi\left(\int_{(-\infty,0]} h(s{-}u)m(\diff u)\right)\right),\diff s\right), t{\ge}0\right)\\
    \steq{dist} \left({\cal R}\left(\left(0,\int_{(0,t]}\Phi\left(\int_{(-\infty,0]} h(s{-}u)m(\diff u)\right)\diff s\right)\right), t\ge 0\right)
\end{multline}
between the processes of counting measures of two point processes is an easy consequence of the equality of their respective Laplace transforms. See Chapter~1 of~\citet{robert_stochastic_2003}. 

If $E_1$ is the first point of ${\cal R}$, we deduce therefore the  identity 
\[
\int_{(0,T_1]}\Phi\left(\int_{(-\infty,0]} h(s{-}u)m(\diff u)\right)\diff s\steq{dist} E_1.
\]
Using Definition~\ref{Tdef}, the random variable $T_1$  has therefore the same distribution as the random variable ${\cal T}(x,E_1)$ where $x$ is the unique element of ${\cal S}_h$  such that $m{=}m_x$ and $E_1$ is an exponential random variable with parameter $1$. It is in particular almost surely finite by Lemma~\ref{Lem1}. 

By induction, for $n{\ge}1$, if $T_1{\le}T_2{\le}\cdots{\le}T_n{<}{+}\infty$, it is possible to define
\begin{multline}\label{TNeq}
T_{n+1}{\steq{def}} \inf\left\{\rule{0mm}{6mm}t{\ge}T_{n}:\rule{0mm}{4mm}\right.\\
\left. \int_{(T_n,t]}{\cal P}\left( \Phi\left(\int_{(-\infty,0]}\hspace{-5mm} h(s{-}x)m(\diff x)\right){+}
\sum_{k=1}^n h\left(s{-}T_k\right), \diff s\right){\ne}0\right\},
\end{multline}

The random variable $T_{n+1}$ as defined by~\eqref{TNeq} is clearly a stopping time, it is almost surely finite with the same argument
as for $T_1$. With the strong Markov property of ${\cal P}$ for the stopping time $T_n$ and Relation~\eqref{eqpp}, but conditionally  on ${\cal F}_{T_n}$, 
\begin{multline*}
  \left.  \left( \int_{(T_n,T_n+t]}{\cal P}\left( \Phi\left(\int_{(-\infty,0]}\hspace{-5mm} h(s{-}x)m(\diff x)\right){+}
\sum_{k=1}^n h\left(s{-}T_k\right),\diff s\right)\right|{\cal F}_{T_n}\right)\\
\steq{dist}  \left.\left({\cal R}\left(\left(0,\int_{T_n}^{T_n+t}\Phi\left(\int_{(-\infty,0]} h(s{-}u)m(\diff u){+}\sum_{k=1}^n h(s{-}T_k)\right)\diff s\right)\right)\right|{\cal F}_{T_n}\right),
\end{multline*}
hence, as for the case $n{=}1$, if $E_{n+1}$ is the distance between the $n$th and $(n{+}1)$th point of ${\cal R}$, it  is an exponential random variable independent of ${\cal F}_{T_n}$ and 
\begin{multline*}
E_{n+1}=\int_{T_n}^{T_{n+1}}\Phi\left(\int_{(-\infty,0]} h(s{-}u)m(\diff u){+}\sum_{k=1}^n h(s{-}T_k)\right)\diff s\\=
\int_{T_{n}}^{T_{n+1}}\Phi\left(\int_{({-}\infty,s)}
h\left(s{-}x\right){\cal N}_m(\diff x)\right)\, \diff s.
\end{multline*}
The  proposition is proved.
\end{proof}
In the proof  we have seen that the condition $\Phi(0){>}0$ in Lemma~\ref{Lem1} guarantees that there is almost surely an
infinite number of points in $\R_+$ for SDE~\eqref{HSDE}.

It cannot be excluded nevertheless that this non-decreasing sequence of points may have a finite limit with positive probability and therefore that ${\cal N}_m$ may not be a point process.  As it will be seen, it can indeed blow-up in finite time, i.e. the limit of its sequence $(T_n)$ of points in $\R_+$ may be finite with positive probability. Proposition~\ref{TnIneq} gives a condition for which the sequence $(T_n)$ converges almost surely to infinity when $m$ is the null measure.  Theorem~\ref{TransientH} considers a case when there is such a blow-up and establishes  a limiting result for the accumulation of these points.

However, when the sequence $(T_n)$ is converging almost surely to infinity, ${\cal N}_m$ is a point process, and it is not difficult to see that the dual predictable projection of the process of $({\cal N}_m((0,t]), t{\ge}0)$ is 
  \[
\left(\int_0^t \Phi\left(\int_{(-\infty,s)} h(s{-}x){\cal N}_m(\diff x)\right)\diff s, t{\ge}0\right).
\]
see Theorem~VI~(27.1) of~\citet{rogers_diffusions_2000}. Consequently, the stochastic intensity function of ${\cal N}_m$ 
on $\R_+$ is indeed
\[
\left(\Phi\left(\int_{(-\infty,s)} h(s{-}x){\cal N}_m(\diff x)\right)\right).
\]

\vspace{1em}

Proposition~\ref{HSDEprop} extends a point measure $m$ on $\R_-$ to a point process on $\R$, it can be also seen as a dynamical system on point processes on $\R_-$ in the following way.
\subsection*{A dynamical system}\label{Nt}
If $m{\in}{\cal M}_p(\R_-)$ is such that $h({-}x){\in}L_1({m})$ and the non-negative measure ${\cal N}_m$ defined by Relation~\eqref{HSDE} is a point process, for $t{\ge}0$, we introduce a  (random) dynamical system $(T_t(m))$ in ${\cal M}_p(\R_-)$  as,
\begin{equation}\label{DynSys}
\int_{(-\infty,0]} f(x)T_t(m)(\diff x)=
  \int f(x)\theta_t({\cal N}_m)(\diff x)= \int_{(-\infty,t]} f(x{-}t){\cal N}_m(\diff x),
\end{equation}
for any non-negative Borelian function on $\R_-$, $T_t(m)$ is the point process ${\cal N}_m$ seen from the point $t$. 

A stationary point process is a distribution $Q$  on ${\cal M}_p(\R)$ which is invariant distribution for the group of transformations $(\theta_t)$.  See Definition~\ref{defiStat} in Section~\ref{AppPP} of the appendix. By Relation~\eqref{HSDE}, it can be formulated as the existence  of a  distribution on ${\cal M}_p(\R_-)$ invariant by the operator $T_t$. 

When $\Phi(0){=}0$, the null measure is clearly a solution of Relation~\eqref{HSDE}. The next proposition shows that, under some mild condition, the null point process is the unique stationary Hawkes process in such a case.
\begin{proposition}\label{deathH}
  Under the condition
  \[
  \int_{\R_+}th(t)\diff t<{+}\infty,
  \]
   if for some $K{\ge}0$, the non-negative function $\Phi$ satisfies the relation $\Phi(x){\le}Kx$,
for all $x{\ge}0$, then there does not exist an ergodic stationary Hawkes process ${\cal N}$  which is non-trivial, i.e. such
that $\P({\cal N}{\not\equiv}0){>}0$.
\end{proposition}
See Definition~\ref{defiStat} for the stationarity and ergodicity properties,  they are with respect to the flow of translation $(\theta_t)$. See~\citet{cornfeld_ergodic_1982}.
\begin{proof}
Assume that ${\cal N}$ is a non-trivial ergodic stationary Hawkes process. The Birkhoff-Khinchin ergodic theorem, see~\cite{cornfeld_ergodic_1982}, gives the almost sure convergence 
\[
\lim_{t{\to}{+}\infty} \frac{{\cal N}((0,t])}{t}=\lambda\steq{def}\E({\cal N}((0,1]){>}0,
\]
since ${\cal N}$ is non-trivial, we obtain therefore that ${\cal N}(\R_+)$ is almost surely infinite. If  $T_1$  is  the first positive point of ${\cal N}$, the variable $T_1$  is in particular finite with probability $1$.
  
As in the proof of Proposition~\ref{HSDEprop}, we have
  \[
\P(T_1{\ge}t\mid {\cal F}_0)=\exp\left(-\int_0^t \Phi\left(\int_{({-}\infty,0]} h(s{-}x){\cal N}(\diff x)\right)\,\diff s\right).
  \]
   The stationary property of the point process gives the relation
   \[
   \E\left(\int_\R f(s){\cal N}(\diff s)\right)=\lambda \int_\R f(s)\diff s,
   \]
   for all $f{\in}{\cal C}_c(\R)$.  If
  \[
  V\steq{def}\int_0^{+\infty} \Phi\left(\int_{({-}\infty,0]} h(s{-}x){\cal N}(\diff x)\right)\,\diff s,
  \]
we have, by Fubini's Theorem,
\begin{multline*}
\E(V) \le K\E\left(\int_0^{+\infty} \int_{({-}\infty,0]} h(s{-}x){\cal N}(\diff x)\,\diff s\right)
\\= \lambda K\int_0^{+\infty} \int_{({-}\infty,0]} h(s{-}x)\,\diff x\diff s=\lambda K\int_0^{+\infty}s h(s)\diff s<{+}\infty,
\end{multline*}
hence $V$ is an integrable random variable, it is in particular almost surely finite. As a consequence we have that, almost surely, $\P(T_1{=}{+}\infty{\mid} {\cal F}_0){>}0$. This is a contradiction. The proposition is proved.
\end{proof}

\section{A Markov Chain Formulation}
\label{MarkSec}
In this section, we give another version of Proposition~\ref{HSDEprop} in terms of a Markovian dependence on the state space of sequences. 

\subsection{An alternative formulation}
We begin with a characterization of the Palm measure of a stationary Hawkes process.

For $m{\in}{\cal M}_p(\R)$, we define $(\tau_n(m),n{\in}\Z){\steq{def}}(t_{n}(m){-}t_{n-1}(m),n{\in}\Z)$, where $(t_n(m))$ is given by Relation~\eqref{eqt1} of the Appendix. As a random variable on ${\cal M}_p(\R)$, the sequence is also represented as $(\tau_n,n{\in}\Z)$. 
\begin{proposition}\label{InterPalmTheo}
If $\Phi(0){>}0$ and if $Q$ is the distribution on ${\cal M}_p(\R)$ of a stationary Hawkes process
associated to $\Phi$ and $h$ then, under its associated Palm measure $\widehat{Q}$, the sequence of inter-arrivals $(\tau_n,n{\in}\Z)$  is a stationary sequence. The sequence of random variables
\begin{equation}\label{InterPalm}
\left(\int_{0}^{\tau_{n+1}}\Phi\left(\sum_{k{\le}n}
h\left(s{+}\sum_{i=k{+}1}^n\tau_i\right)\right)\,\diff s,\quad n{\in}\Z\right),
\end{equation}
is i.i.d. with a common exponential distribution with parameter $1$.

If there exists a stationary sequence $(\tau_n,n{\in}\Z)$ satisfying Relation~\eqref{InterPalm}, then there exists a
stationary Hawkes process associated to $\Phi$ and $h$.
\end{proposition}
From the point of view of Proposition~\ref{HSDEprop} this proposition is quite intuitive, one has nevertheless to be
careful, as always in this setting, since Proposition~\ref{HSDEprop} is stated under the distribution $Q$.
\begin{proof}
For simplicity, in this proof we will use the canonical probability space $\Omega{=}{\cal M}_p(\R)$ endowed with the weak topology and $\E_R$ denotes the expectation with respect to a probability distribution $R$ on ${\cal M}_p(\R)$.

The stationary property of  the sequence $(\tau_n,n{\in}\Z)$ under $\widehat{Q}$ is clear, by definition of $\widehat{Q}$.
For $t{\ge}0$, if $f$ is a bounded Borelian function on ${\cal M}_p(\R)$, we denote
    \[
{\cal E}_t \steq{def} \{m\in {\cal M}_p(\R): m([-t,0]){\not=}0\}
    \]
then, Proposition~11.8, page~315 of \citet{robert_stochastic_2003} shows that if $t{\mapsto}f(\theta_t(m))$ is $Q$-almost surely right continuous at $t{=}0$, then the relation
    \[
    \lim_{t\searrow 0} \E_Q\left(f \mid {\cal E}_t\right) =E_{\widehat{Q}}(f)
    \]
holds. 

For $n{\in}\Z$ and $m{\in}{\cal M}_p(\R)$, define
    \[
    \Psi_n(m)\steq{def} \int_{t_n(m)}^{t_{n+1}(m)} \Phi\left(\int_{({-}\infty,s]} h(s{-}x)m(\diff x)\right)\,\diff s,
    \]
since, $Q$-almost surely,  $t_0(m){<}0{<}t_1(m)$ and $t_n(m){<}t_{n+1}(m)$, then for $t{\ge}0$ such that $0{\le}t{<}t_1(m)$, we have
    \[
    \Psi_n(\theta_t(m))=\int_{t_n(m)-t}^{t_{n+1}(m)-t} \Phi\left(\int_{({-}\infty,s{+}t]} h(s{+}t{-}x)m(\diff x)\right)\,\diff s=\Psi_n(m).
      \]
Let $F$ be a continuous bounded Borelian function on $\R_+^n$ then
    \[
    \lim_{t\searrow 0} \left.\E_Q\left(F\left(\rule{0mm}{4mm}(\Psi_i
       ,1{\le}i{\le}n)\right)\right| {\cal E}_t\right)
    = \E_{\widehat{Q}}\left(F\left(\rule{0mm}{4mm} (\Psi_i,1{\le}i{\le}n)\right)\right).
    \]
Since the event $ {\cal E}_t$ is ${\cal F}_0$-measurable  and that for any $1{\le}k{\le}n$, the random variable
    \[
    m\longrightarrow\int_{t_{k}(m)}^{t_{k+1}(m)}\Phi\left(\int_{({-}\infty,s]} h\left(s{-}x\right)m(\diff x)\right)\,\diff s
    \]
is independent of ${\cal F}_{t_k}$ and therefore of ${\cal F}_0$, by Proposition~\ref{HSDEprop}, by induction on $k$ for example, we obtain that
    \[
    \lim_{t\searrow 0} \left.\E_Q\left(F\left(\rule{0mm}{4mm}
        (\Psi_i,1{\le}i{\le}n)\right)\right| {\cal E}_t\right)=
    \E\left(F\left(\rule{0mm}{4mm} (E_i,1{\le}i{\le}n)\right)\right),
    \]
where $(E_i,1{\le}i{\le}n)$ are $n$ independent random variables with a common exponential distribution with parameter
$1$. By gathering these results we obtain that
  \[
  \E_{\widehat{Q}}\left(F\left(\rule{0mm}{4mm}(\Psi_i,1{\le}i{\le}n)\right)\right)=
  \E\left(F\left(\rule{0mm}{4mm} (E_i,1{\le}i{\le}n)\right)\right).
  \]
Under the probability distribution $\widehat{Q}$, the random variables $\Psi_i$, $1{\le}i{\le}n$, are i.i.d.
with a common exponential distribution with parameter~$1$. Relation~\eqref{InterPalm} is a consequence of the relation, for $m{\in}{\cal M}_p(\R)$, 
\begin{multline*}
\Psi_n(m)= \int_{t_n(m)}^{t_{n+1}(m)}\Phi\left(\sum_{k{\le}n} h\left(s{-}t_k(m)\right)\right)\,\diff s
\\ = \int_{0}^{\tau_{n+1}(m)}\Phi\left(\sum_{k{\le}n} h\left(s{+}\sum_{i=k{+}1}^n\tau_i(m)\right)\right)\,\diff s,
\end{multline*}
hence $\Psi_{n+1}(m)$ is the variable $\Psi_n(m)$ associated to the sequence $(\tau_{n+1}){=}\widehat{\theta}(\tau_n)$, see Appendix~\ref{AppPP}. 
By invariance of $\widehat{Q}$ with respect to $\widehat{\theta}$, see Proposition~11.18 of~\citet{robert_stochastic_2003} for example, the sequence $(\Psi_n, {n{\in}\Z})$ is stationary  under  $\widehat{Q}$  and, therefore, it is  i.i.d. with a common exponential distribution with parameter~$1$.

 Now we assume that  there exists a stationary sequence ${\tau}{=}(\tau_n,n{\in}\Z)$ of integrable random variables satisfying Relation~\eqref{InterPalm}.
 Recall that $m_\tau$ is the point process defined by Relation~\eqref{mx}. Using a similar proof as in Proposition~\ref{HSDEprop},
 we can show that $m_\tau$ satisfies an Hawkes SDE~\eqref{HSDE} on $\R$.
 For $u{\in}\R$, the same property clearly  holds for $\theta_u(m_\tau)$, the point process $m_\tau$ translated at $u$, see Definition~\eqref{thetat}. For $K{>}0$, let $U_K$ be an independent uniform random variable on $[-K,K]$, the point process $\theta_{U_K}(m_\tau)$ satisfies the Hawkes property.  With the same method used for the proof of Proposition~11.2 of~\cite{robert_stochastic_2003}, we obtain that, as $K$ goes to infinity, $\theta_{U_K}(m_\tau)$ converges in distribution to a stationary point process whose Palm measure is given by the distribution $\tau$.  It is not difficult to show that the Hawkes property is preserved in the limit. The proposition is proved.
\end{proof}

\subsection{A Markov chain on $\mathbf{{\cal S}}$}

The previous proposition has highlighted the importance of the Palm space, and therefore led us to develop
a second formulation of the Hawkes property using a Markovian kernel. 

\begin{definition}\label{MarkDef}
The sequence of random variables $({\cal X}_n^x)$ with initial point ${\cal X}_0^x{=}x{\in}{\cal S}$ is defined by
induction as follows, for $n{\ge}0$,
\begin{equation}\label{Trans}
{\cal X}_{n+1}^x{=}(X^x_{n{+}1},{\cal X}_n^x) 
\end{equation}
where $X^x_{n{+}1}$ is defined by the relation
\begin{equation}\label{eqX0}
 \int_{0}^{X^x_{n+1}} \Phi\left( \int_{\R_-}h(T_{n}{+}s{-}u)m(\diff u){+}
        \sum_{k{=}1}^{n} h\left(s{+}\sum_{i=k+1}^{n}X^x_i \right)\right)\,\diff s=E_{n+1},
\end{equation}
with $T_{n}{=}X^x_1{+}\cdots{+}X^x_{n}$ and  $(E_k)$ is an  i.i.d. sequence of exponentially distributed random variable with parameter $1$.

The associated Markovian kernel is denoted by ${\cal K}$
\begin{equation}\label{Kernel}
\int_{\cal S}f(y) {\cal K}(x,\diff y)= \E_x(f({\cal X}_1^x)),
\end{equation}
for a non-negative Borelian function $f$ on ${\cal S}$.
\end{definition}
The element  ${\cal X}_{n+1}^x$ is obtained by shifting ${\cal X}_n^x$ and adding
$X^x_{n{+}1}{=}{\cal T}({\cal X}_n^x,E_{n+1})$ at the beginning of the sequence. ${\cal T}$
is the operator defined by Relation~\eqref{Tau}.

The sequence  $({\cal X}_n^x)$ clearly has the Markov property.

\begin{proposition}[The Markov chain $({\cal X}_n^x)$ and Hawkes SDEs]\label{HMCprop}
If $\Phi(0){>}0$ and  $m{\in}{\cal M}_p(\R_-)$ is such that $m(\{0\}){>}0$, then the distribution of ${\cal N}_m$, the solution of Relation~\eqref{HSDE} can be expressed as
  \[
  {\cal N}_m\steq{dist} m + \sum_{n{\ge}1} \delta_{T_n},
  \]
with $T_0{=}0$ and, for $n{\ge}1$, $T_{n+1}{-}T_n{=}X^x_{n+1}$,    where the sequence $(X^x_n)$ is defined by Relation~\eqref{eqX0}.

The process   $({\cal X}_n^x){=}((X^x_n,\ldots,X^x_1,{\cal X}_0^x))$  is the Markov chain with transition kernel ${\cal K}(\cdot,\cdot)$ of
Relation~\eqref{Kernel} and initial point
  \[
{\cal X}_0^x{=}(s_{n{+}1}{-}s_n), \text{ if } m{=}(s_n, n{\ge}0)),
\]
with the convention $\cdots{\le}s_n{\le}s_{n-1}{\le}\cdots{\le}s_1{\le}s_0{=}0$. 
\end{proposition}
In~\eqref{eqX0}, recall the convention of Relation~\eqref{Conv}, for $s{\ge}0$ and $n{\ge 2}$, 
\[
\sum_{k{=}1}^{n-1} h\left(s{+}\sum_{i=k+1}^{n-1}X^x_i \right)=
h(s){+} \sum_{k{=}1}^{n-2} h\left(s{+}\sum_{i=k+1}^{n-1}X^x_i \right).
\]
\begin{proof}
This is a straightforward consequence that Relation~\eqref{eqX0} is a rewriting of Relation~\eqref{eqX} of
Proposition~\ref{HSDEprop}.
\end{proof}

We can now state the main result concerning the relation between the Hawkes SDE and the Markov transition kernel ${\cal K}$.
\begin{proposition}\label{PalmMarkov}
The  Markov chain associated to transition kernel ${\cal K}$ of Definition~\eqref{Kernel} has an invariant distribution on ${\cal S}_h$ if and only if there exists a stationary Hawkes process associated to the functions $\Phi$ and $h$.

The distribution of the sequence of inter-arrival times of a stationary Hawkes process associated to the functions $\Phi$ and $h$ is an invariant measure for the  Markov chain associated to transition kernel ${\cal K}$. 
\end{proposition}
\begin{proof}
  If there exists a stationary Hawkes point process $(t_n)$, Proposition~\ref{InterPalmTheo} shows that the distribution of $(t_{n+1}{-}t_n)$ under its Palm distribution is an invariant distribution of the Markov chain $({\cal X}_n)$. Conversely,  if the Markov chain associated to transition kernel ${\cal K}$ has an invariant distribution, then one can construct a stationary version of the Markov chain  ${\cal X}{\steq{def}}({\cal X}_n)$, in particular ${\cal X}$ satisfies Relation~\eqref{InterPalm}.  Proposition~\ref{InterPalmTheo} shows then that there exists a stationary Hawkes point process in this case.
\end{proof}

\subsection{Markov Chains Starting from the Empty State}
We define  Markov chains $({\cal X}^0_n)$ with transition kernel ${\cal K}$ defined by Relation~\eqref{Kernel} when the initial state  empty, i.e.  it is the constant sequence equal to ${+}\infty$, i.e.\ ${\cal X}_0^0{=}({+}\infty)$. This initial state corresponds to the case of  a point process with a point at $0$ only.
The sequence $({\cal X}^0_n)$ is in fact associated to the point process ${\cal N}_{\delta_0}$ of Proposition~\ref{HSDEprop}.

The Markov chain  can be defined by, for  $n{\ge}1$,
\[
 {\cal X}_n^0=(X_n^0,X_{n-1}^0,X_{n-2}^0,\ldots,X_2^0,X_{1}^0,{+}\infty),
\]
where the sequence  $(X_n^0)$ is defined by Relation~\eqref{eqX0} with $m{=}\delta_0$. 
\begin{proposition}\label{TnIneq}
If  $\Phi(0){>}0 $ and $\alpha$ and $\beta$ are defined by Relations~\eqref{alpha} and~\eqref{beta},   there exist $\nu{>}0$  and $\beta_0{>}\beta$ such that,  almost surely, for all $n{\ge}1$, 
\begin{equation}\label{eqT0}
\sum_{k=1}^n \left(E_{k}{-}\alpha\beta_0\right)\le \nu \sum_{k=1}^n X_{k}^0{-}\beta_0\sum_{k=1}^n \overline{H}\left(\sum_{j=k}^n X_j^0\right),
\end{equation}
with 
\[
\overline{H}(x){=}\int_x^{+\infty} h(s)\,\diff s,
\]
If $\alpha\beta{<}1$,  the random measure ${\cal N}_{\delta_0}$ solution of Relation~\eqref{HSDE} of Proposition~\ref{HSDEprop} is almost surely a point process. 
\end{proposition}
\begin{proof}
For any  $\beta_0{>}\beta$, there exists some $\nu{>}0$ such that  Relation $\Phi(x){\le}\nu{+}\beta_0x$, holds for all $x{\ge}0$.  Equation~\eqref{eqX0} gives, for $n{\ge}1$,
\begin{multline}    \label{ineq1}
   E_n=\int_{0}^{X_n^0}\Phi\left(\sum_{k=1}^{n} h\left(s{+}\hspace{-1mm}\sum_{i=k}^{n-1} X_i^0\right)\right)\,\diff s
\\   \le \nu X_n^0+ \beta_0 \int_0^{X_n^0}\sum_{k=1}^{n} h\left(s{+}\hspace{-1mm}\sum_{i=k}^{n-1} X_i^0\right)\,\diff s,
\end{multline}
and, therefore, with the definition of $\overline{H}$, 
\begin{align*}
 E_{n} &\le \nu X_n^0{+}\beta_0 \int_0^{X_n^0} h\left(s\right)\,\diff s{+}
 \beta_0 \int_0^{X_n^0}\sum_{k=1}^{n-1} h\left(s{+}\hspace{-1mm}\sum_{i=k}^{n-1} X_i^0\right)\,\diff s,\\
 E_{n}{-}\alpha\beta_0&\le \nu X_n^0{-}\beta_0 \overline{H}(X_n^0){+}
        \beta_0 \int_{X_{n-1}^0}^{X_{n-1}^0+X_n^0}
            \sum_{k=1}^{n-1} h\left(s{+}\sum_{i=k}^{n-2} X_i^0\right)\,\diff s.
\end{align*}
By using  Inequality~\eqref{ineq1} for the index $n{-}1$  and by adding these relations, we get
\begin{multline*}
  \sum_{k=n-1}^n \left(E_{k}{-}\alpha\beta_0\right) \le  \nu\left[ X_{n}^0{+} X_{n-1}^0\right]
      {-}\beta_0 \left[\overline{H}(X_n^0){+}\overline{H}(X_n^0{+}X_{n-1}^0)\right]
\\ +\beta_0\int_{X_{n-2}^0}^{X_{n-2}^0{+}X_{n-1}^0{+}X_n^0}
    \sum_{k=0}^{n-2} h\left(s+\sum_{i=k+1}^{n-3} X_i^0\right)\,\diff s.
\end{multline*}
By proceeding by induction, we finally get the relation, for $0{\le}p{<}n$,
\begin{multline*}
\sum_{k=n-p}^n \left(E_{k}{-}\alpha\beta_0\right)
\le \nu \sum_{k=n-p}^n X_{k}^0 {-}\beta_0\sum_{k=n-p}^n \overline{H}\left(\sum_{j=k}^n X_j^0\right)\\
 +\beta_0\int_{X_{n-p-1}^0}^{X_{n-p-1}^0{+}\cdots{+}X_n^0}
 \sum_{k=1}^{n-p-1} h\left(s+\sum_{i=k}^{n-p-2} X_i^0\right)\,\diff s,
\end{multline*}
for $p{=}n{-}1$ it gives Relation~\eqref{eqT0}. 

If $\alpha\beta{<}1$, $\beta_0$ can be chosen so that  $\alpha\beta_0{<}1$ for some $\nu{>}0$, hence, by the law of large numbers
\[
\liminf_{n\to+\infty}\frac{1}{n}\sum_{k=1}^n X_k^0\geq \frac{1{-}\alpha\beta_0}{\nu}>0, 
\]
holds almost surely, which implies that ${\cal N}_{\delta_0}$ is almost surely a point process. The proposition is proved. 
\end{proof}

\section{Affine Activation Function}\label{SecCrump}
In this section, we present quickly in a self-contained way  a well-known existence result for a stationary Hawkes process  when the activation function is affine,
\begin{equation}\label{phib}
\Phi_{\beta}(x)=\nu{+}\beta x, \quad x{\ge}0,
\end{equation}
where $\nu$, $\beta{>}0$. This is the model investigated in the classical reference~\citet{hawkes_cluster_1974}. The results are presented in our Markovian setting $({\cal X}_n)$.  It is an important ingredient in the proof of existence of stationary Hawkes processes of Section~\ref{ExisSec}. 

The associated Hawkes process can be interpreted in terms of birth instants of a branching process: births of external individuals occur according to a Poisson process with rate $\nu$, each individual with age $t{>}0$ gives birth to a new individual at rate $\beta h(t)$, the quantity $\alpha\beta$ turns out to be the average number of births due to an individual. Recall that $\alpha$ is the integral of $h$ on $\R_+$, see Relation~\eqref{alpha}. This  branching characterization of the model will not be explicitly used in our arguments. 

We denote by ${\cal Y}_n^0{=}(Y^0_n,Y^0_{n-1},\ldots,Y^0_1,{+}\infty)$ the Markov chain of Definition~\ref{MarkDef} for the function $\Phi{=}\Phi_{\beta}$, whose initial state is empty, i.e. ${\cal Y}_0^0{=}({+}\infty)$. For any $n{\ge}1$, we have 
\begin{equation}\label{eqY0}
 \int_{0}^{Y^0_n} \Phi\left( \sum_{k{=}1}^{n} h\left(s{+}\sum_{i=k+1}^{n-1}Y^0_i \right)\right)\,\diff s=E_n,
\end{equation}
where $(E_n)$ is an i.i.d. sequence of exponential random variables with parameter $1$. 

\begin{proposition}\label{CrumpProp}
Under the condition $\alpha\beta{<}1$, the sequence $({\cal Y}_n^0)$ is converging in distribution, for the topology induced by $d$ of Relation~\eqref{DistS},  to a stationary sequence ${\cal Y}^\infty{=}(Y_1^\infty,Y_2^\infty,\ldots,Y_n^\infty,\ldots)$ such that
  \[
  \left(\int_0^{{Y}_{n}^\infty}\left(\Phi_{\beta}\left(\sum_{i=n}^{+\infty} h\left(u{+}\sum_{j=n+1}^{i}{Y}_{j}^\infty\right)\right)\diff u\right)\right),
  \]
    is an i.i.d. sequence of exponential random variables with parameter $1$ and
    \[
    \E\left(Y_1^\infty\right)=\frac{1{-}\alpha\beta}{\nu},
    \]
    furthermore, if 
    \[
(I_n)\steq{def} \left(\sum_{k=1}^{n} h\left(\sum_{i=2}^kY_i^n\right)\right) \text{ and } I\steq{def}  \sum_{k=1}^{{+}\infty} h\left(\sum_{i=2}^kY_i^\infty\right),
    \]
    the sequence $(I_n)$ converges in distribution to the integrable random variable $I$. 
\end{proposition}
The proposition gives in particular that  $Y_1^\infty{\steq{dist}}{\cal T}((Y_n^{\infty},n{\ge}2),E_1)$. The variable $\nu{+}\beta{I}$ is the intensity function of the stationary Hawkes process just after one of its point, i.e. for its Palm distribution.

As it can be seen, compared to ${\cal Y}^\infty$, the numbering of the coordinates of ${\cal Y}_n^0$ is reversed. This is mainly for a notational convenience in fact. The important fact is that the sums inside the function $\Phi$ in both cases are depending on the past. 
\begin{proof}
The proof is done in several steps. 

  \medskip
  \noindent
  {\sc Step~1}. Backward Coupling.
  
The idea is of using a backward coupling idea. See Chapter~22 of~\citet{Levin} for a general presentation and~\citet{Loynes} for an early use of this method.

 We define a sequence of random variables   $(\widetilde{\cal Y}_n^0){=}(\widetilde{Y}_1^n,\ldots,\widetilde{Y}_n^n,{+}\infty)$ by induction as follows, recall that $(E_n)$ is an i.i.d. of random variables with an exponential distribution with parameter $1$, 
  \[
 \begin{cases}
\widetilde{Y}_n^n&{=} {\cal T}(({+}\infty),E_{n}),\\
\widetilde{Y}_{n-k}^n&{=} {\cal T}\left(\left(\widetilde{Y}_{n-k+1}^n,\ldots,\widetilde{Y}_n^n,{+}\infty\right),E_{n-k}\right), \quad 1{\le}k{\le}n,
 \end{cases}
 \]
in particular,
 \[
 \widetilde{Y}_{1}^n{=} {\cal T}\left(\left(\widetilde{Y}_{2}^n,\ldots,\widetilde{Y}_n^n,{+}\infty\right),E_{1}\right).
 \]
 It is easily checked that
\begin{equation}\label{eqavx1}
\left(\widetilde{Y}_1^n,\ldots,\widetilde{Y}_n^n,{+}\infty\right)\steq{dist} \left(\widetilde{Y}_2^{n+1},\ldots,\widetilde{Y}_{n+1}^{n+1},{+}\infty\right)
\end{equation}
and that
\begin{equation}\label{eqavx2}
\widetilde{\cal Y}_n^0\steq{dist}{\cal Y}_n^0, \text{ since }  (E_1,E_2,\ldots,E_n)\steq{dist}  (E_n,E_{n-1},\ldots,E_1). 
\end{equation}
For any $n{\ge}1$ and $1{\le}k{\le}n$, we show that $\widetilde{Y}_k^{n+1}{\le}\widetilde{Y}_k^n$ holds.
This is done by induction. Since $h$ is non-negative and non-increasing, we have
\[
E_n= \int_0^{{\widetilde{Y}_{n}^n}}\left(\nu{+}\beta h\left(u\right)\right)\diff u\leq \int_0^{{\widetilde{Y}_{n}^n}}\left(\nu{+}\beta\left(h(u){+} h\left(u{+}\widetilde{Y}_{n+1}^{n+1}\right)\right)\right)\diff u,
\]
and, since,
\[
E_{n+1}=\int_0^{{\widetilde{Y}_{n}^n}}\left(\nu{+}\beta\left(h(u){+} h\left(u{+}\widetilde{Y}_{n+1}^{n+1}\right)\right)\right)\diff u,
\]
this implies that ${\widetilde{Y}_{n}^n}{\ge}\widetilde{Y}_{n}^{n+1}$ holds.
If the relation ${\widetilde{Y}_{i}^n}{\ge}\widetilde{Y}_{i}^{n+1}$ holds for any $k{+}1{\le}i{\le}n$, then 
\begin{multline}\label{eqavx5}
E_k{=} \int_0^{{\widetilde{Y}_{k}^n}}\hspace{-1mm}\left(\nu{+}\beta\sum_{i=k}^{n} h\left(u{+}\hspace{-2mm}\sum_{j=k+1}^{i}\widetilde{Y}_{j}^n\right)\right)\diff u\\ \leq \int_0^{{\widetilde{Y}_{k}^n}}\hspace{-1mm}\left(\nu{+}\beta\sum_{i=k}^{n} h\left(u{+}\hspace{-2mm}\sum_{j=k+1}^{i}\widetilde{Y}_{j}^{n+1}\right)\right)\diff u,
\end{multline}
which gives the relation  ${\widetilde{Y}_{k}^n}{\ge}\widetilde{Y}_{k}^{n+1}$. We can now define, for $k{\ge}1$, 
\[
\widetilde{Y}_{k}^\infty=\lim_{n\to+\infty}\widetilde{Y}_{k}^n.
\]
Relation~\eqref{eqavx1} shows that  the sequence $(\widetilde{Y}_{k}^\infty)$ is stationary and Relation~\eqref{eqavx2} gives the desired convergence in distribution for $({\cal Y}_n^0)$. 

\medskip
\noindent
    {\sc Step~2}. Convergence of the intensities $(I_n)$.
Proposition~\ref{TnIneq} shows that, for all $1{\le}i{\le}n$,
\begin{equation}\label{eqavx3}
\sum_{j=2}^{i}\widetilde{Y}_{j}^n\ge S_i\steq{def} \frac{1}{\nu}\sum_{j=2}^{i}\left(E_j{-}\alpha\beta\right).
\end{equation}
The assumption on $\alpha\beta$ shows that, almost surely $(S_i)$ is converging to infinity. If $a^+{=}\max(a,0)$, then
\[
\sum_{i=2}^{n} h\left(\sum_{j=2}^{i}\widetilde{Y}_{j}^n\right)\le \sum_{i=2}^{n} h\left(S_i^+\right). 
\]
Since $h$ is non-increasing converging to $0$ at infinity, there exists a non-negative finite Radon measure $\mu$ on $\R_+$ such that
$h(x){=}\mu((x,{+}\infty))$,
with Fubini's Theorem, we obtain
\begin{equation}\label{eqcl}
\E\left(\sum_{i=1}^{+\infty} h\left(S_i^+\right)\right)= \int_0^{+\infty}\E\left( {\cal N}_+(0,x)\right)\mu(\diff x),
\end{equation}
where ${\cal N}_+$ is the counting measure of the sequence $(S_n)$, i.e.  for $x{\ge}0$, 
\begin{equation}\label{Cdef}
{\cal N}_+(0,x)\steq{def} \sum_{i=2}^{+\infty}\ind{x{\ge} S_i}.
\end{equation}
The Laplace transform of ${\cal N}_+$ on $\R_+$ is given by, for $\xi{\ge}0$ sufficiently small so that $\nu\exp(\alpha\beta \xi/\nu){<}(\nu{+}\xi)$, we have
\[
\omega(\xi)\steq{def}\E\left(\sum_{n{\ge}1} e^{-\xi S_n}\right)=\frac{\nu{+}\xi}{\nu{+}\xi{-}\nu\exp(\alpha\beta\xi/\nu)}
\]
Hence, 
\[
\lim_{\xi\to 0}\xi\omega(\xi)=\frac{\nu}{1{-}\alpha\beta},
\]
A Tauberian theorem, see Theorem~2 of Chapter~XIII of~\citet{Feller},  implies that
\[
\lim_{x\to+\infty} \frac{\E({\cal N}_+(0,x))}{x}=\frac{\nu}{1{-}\alpha\beta},
\]
hence there exist $D_1$ and $D_2{\ge}0$ such that for all $x{\ge}0$, 
\begin{equation}\label{eqavx4}
\E({\cal N}_+(0,x))\leq D_1{+}D_2x.
\end{equation}
Hence, with Relation~\eqref{eqcl}
\begin{multline}\label{Cineq}
  \E\left(\sum_{i=1}^{+\infty} h\left(S_i^+\right)\right)\le D_1h(0){+} D_2\int_0^{+\infty}x\mu(\diff x)\\
  \le D_1h(0){+}D_2\int_0^{+\infty}\mu(x,{+}\infty)\diff x
  =D_1h(0){+}D_2\alpha,
\end{multline}
with the integration by parts formula.

For $p{\ge}1$ and $n{\ge}p$, 
\[
\sum_{i=p}^{n} h\left(\sum_{j=2}^{i}\widetilde{Y}_{j}^n\right)
\le  \sum_{i=p}^{n}h\left(S_i^+\right)\le \sum_{i=p}^{+\infty}h\left(S_i^+\right),
\]
this last term converges almost surely to $0$ as $p$ gets large, hence we obtain the almost sure convergence
\[
\lim_{n\to+\infty} \sum_{i=1}^{n} h\left(\sum_{j=2}^{i}\widetilde{Y}_{j}^n\right)=\sum_{i=1}^{+\infty} h\left(\sum_{j=2}^{i}\widetilde{Y}_{j}^\infty\right)
\le\sum_{i=1}^{+\infty} h\left(S_i^+\right)<{+}\infty, a.s.
\]
by Relations~\eqref{eqavx3} and~\eqref{eqavx4}.

\medskip
\noindent
    {\sc Step~3}.
     Relation~\eqref{eqavx5} with $k{=}1$ gives
    \[
 E_1{=}\int_0^{{\widetilde{Y}_{1}^1}}\ind{u{\le}{\widetilde{Y}_{1}^n}}\left(\nu{+}\beta\sum_{i=1}^{n} h\left(u{+}\hspace{-2mm}\sum_{j=2}^{i}\widetilde{Y}_{j}^n\right)\right)\diff u,
 \]
 since
\begin{multline*}
 \ind{u{\le}{\widetilde{Y}_{1}^n}}\left(\nu{+}\beta\sum_{i=1}^{n} h\left(u{+}\hspace{-2mm}\sum_{j=2}^{i}\widetilde{Y}_{j}^n\right)\right)
 \le \nu{+}\beta\sum_{i=1}^{n} h\left(\sum_{j=2}^{i}\widetilde{Y}_{j}^n\right)
\\ \le\nu{+}\beta \sum_{i=1}^{+\infty} h\left(S_i^+\right) <{+}\infty,
\end{multline*}
with the same argument as in the proof of the convergence $(I_n)$, and with Lebesgue's dominated convergence theorem, we obtain the relation
    \[
E_1{=} \int_0^{{\widetilde{Y}_{1}^\infty}}\hspace{-1mm}\left(\nu{+}\beta\sum_{i=1}^{+\infty} h\left(u{+}\sum_{j=2}^{i}\widetilde{Y}_{j}^\infty\right)\right)\diff u.
\]
With a  standard argument, via Kolmogorov's extension Theorem, see~Section~VII-38 of~\citet{Halmos},  there exists a stationary sequence $(\widetilde{Z}_{n}^\infty,n{\in}\Z)$ of random variables such that
$(\widetilde{Z}_{n}^\infty,n{\ge}p){\steq{dist}}(\widetilde{Y}_{n}^\infty,n{\ge}1)$, holds for all $p{\in}\Z$. 
After integration of the last relation, with a change of variable,  we obtain
\[
1{-}\nu\E\left(Y_1^{\infty}\right)=\beta\int_0^{+\infty} h\left(u\right)\sum_{i=1}^{+\infty}\P\left(\sum_{j=2}^{i}\widetilde{Y}_{j}^\infty \le u\le \sum_{j=1}^{i}\widetilde{Y}_{j}^\infty\right)\diff u.
\]
For $i{\ge}1$, we have
\begin{multline*}
\P\left(\sum_{j=2}^{i}\widetilde{Y}_{j}^\infty \le u\le \sum_{j=1}^{i}\widetilde{Y}_{j}^\infty\right)=
\P\left(\sum_{j=-i+1}^{-1}\widetilde{Z}_{j}^\infty \le u\le \sum_{j=-i}^{-1}\widetilde{Z}_{j}^\infty\right)\\=
\P\left(\sum_{j=1}^{i-1}\widetilde{Z}_{-j}^\infty \le u\le \sum_{j=1}^{i}\widetilde{Z}_{-j}^\infty\right),
\end{multline*}
hence
\[
  1{-}\nu\E\left(Y_1^{\infty}\right)
  =  \beta\P\left(\int_0^{+\infty} h\left(u\right)\sum_{i=1}^{+\infty}\ind{\sum_{j=1}^{i-1}\widetilde{Z}_{-j}^\infty \le u\le \sum_{j=1}^{i}\widetilde{Z}_{-j}^\infty}\right) \diff u
=\beta\alpha,
\]
by Fubini's Theorem. The proposition is proved. 
\end{proof}
\begin{proposition}
The Markov chain with transition kernel ${\cal K}$ on ${\cal S}_h$ has a unique invariant measure.
\end{proposition}
\begin{proof}
Let $\Lambda$ a probability distribution on ${\cal S}_h$ invariant by ${\cal K}$, and ${\cal Z}{=}(Z_{-p},p{\ge}0)$ a random variable on ${\cal S}$ with distribution $\Lambda$. We construct the sample paths of Markov chains starting $({\cal Z}_n){=}(Z_{n},Z_{n-1},\ldots,Z_1){\cdot}(Z_{-p},p{\ge}0)$ and $({\cal Y}_n^0)$, from initial states ${\cal Z}$ and $({+}\infty)$ respectively, in the following way, for $n{\ge}1$, 
  \[
  \int_0^{Z_{n}}\left(\nu{+}\beta\left(\sum_{i=0}^{n}h\left(s{+}\sum_{j=i}^{n-1} Z_{j}\right){+} \sum_{i=0}^{+\infty}h\left(s{+}\sum_{j=0}^{i} Z_{-j}\right)\right)\right)\diff s =E_{n},
  \]
  and
  \[
  \int_0^{Y_{n}}\left(\nu{+}\beta \sum_{i=0}^{n}h\left(s{+}\sum_{j=i}^{n-1} Y_{j}\right)\right)\diff s =E_{n}.
  \]
  It is easily seen by induction that, almost surely, for all $n{\ge}1$, the relation $Z_n{\le}Y_n^0$ holds, hence, since ${\cal Z}_n{\steq{dist}}{\cal Z}$, for any  $k{\ge}1$ and $n{\ge}k$,
  \[
  (Z_0,\ldots ,Z_{-k})\le_{\rm st} (Y_{n}^0,\ldots ,Y_{n-k}^0),
  \]
  where $\le_{st}$ denotes the stochastic order with respect to the order of the coordinates. Proposition~\ref{CrumpProp} gives therefore
    \[
  (Z_0,\ldots ,Z_{-k})\le_{\rm st} (Y_{0}^\infty,\ldots ,Y_{k}^\infty).
    \]
 Since $(Z_n)$ is invariant for ${\cal K}$, the relation ${\cal T}({\cal Z},E_0){\steq{dist}}Z_1$ holds and,  as in the proof of Proposition~\ref{CrumpProp} the relation $\E(Z_1){=}(1{-}\alpha\beta)/\nu$  also holds, and therefore, for all $k{\ge}1$,  
    \[
  (Z_0,\ldots ,Z_{k})\steq{dist} (Y_{0}^\infty,\ldots ,Y_{k}^\infty).
    \]
The proposition is proved. 
\end{proof}

\section{Stationary Hawkes Processes}\label{ExisSec}
We start with a coupling result between a general non-linear Hawkes process and the Hawkes process of~\citet{hawkes_cluster_1974}. The idea of the coupling with the ``classical'' Hawkes process is also mentioned in~\citet{karabash_stability_2012} but in a space of functions. The existence result of this reference seems to be incomplete since the topology used is not clearly defined and an important continuity property, the Feller property in fact,  is apparently missing. This is the main technical difficulty of the proof of Theorem~\ref{ExistenceTheo}. 
\subsection{A Monotonicity Property}\label{MonSec}
Throughout this section, $(E_n)$ denotes an i.i.d. sequence of exponential random variables with parameter $1$.

Let $m$ be a point measure on $\R_-$, with Definition~\ref{MarkDef}, the sequences $(X^m_n,n{\ge}1)$ and  $(Y^m_n,n{\ge}1)$ are defined by induction by, for $n{\ge}1$,
\begin{align}
\int_0^{X^m_n} \Phi\left(\int_{-\infty}^0 h\left(s{+}\sum_{i=1}^{n-1} X^m_i{-}x\right)m(\diff x){+}\sum_{k{=}0}^{n-1} h\left(s{+}\hspace{-3mm}\sum_{i=n-k}^{n-1} X^m_i\right)\right)\,\diff s=E_n,\label{X0}\\
\int_0^{Y^m_n}  \Phi_{\beta_0}\left(\int_{-\infty}^0 h\left(s{+}\sum_{i=1}^{n-1} Y^m_i{-}x\right)m(\diff x){+}\sum_{k{=}0}^{n-1} h\left(s{+}\hspace{-3mm}\sum_{i=n-k}^{n-1} Y^m_i\right)\right)\,\diff s=E_n, \label{Y0}
\end{align}
where $\beta_0$ is, for the moment, a positive constant and $\Phi_{\beta_0}(x){\steq{def}}\nu{+}\beta_0x$ for $x{\in}\R$.

The affine function $\Phi_{\beta_0}$  will be used as follows. If $\beta$ defined by Relation~\eqref{beta}  is such that $\alpha\beta{<}1$, then there exists some $\beta_0{>}\beta$ such that $\alpha\beta_0{<}1$ and  a constant $\nu$ such that

\begin{proposition}[Coupling]\label{propmono}
If $\nu$ and $\beta_0{>}0$ are such that  the relation $\Phi(x){\le}\Phi_{\beta_0}(x)$, for all $x{\ge}0$, with $\Phi_{\beta_0}$ defined by Relation~\eqref{phib}, and if $m$ be a point measure on $\R_-$ with
\[
\int_{-\infty}^0 h({-}x)m(\diff x){<}{+}\infty,
\]
then, for  $\beta_0 {>} \beta$ there exists $\nu > 0$ such that, almost surely,  the sequence $(Y_n^m)$ defined by Relation~\eqref{Y0}  satisfies the relation
$Y_n^m {\le}X_n^m$,  for  all $n{\ge}1$.
\end{proposition}
\begin{proof}
The proof is  straightforward. If, by induction,  the relation holds up to index $n{-}1{>}0$. For $t{\ge}0$,
\begin{align*}
\int_0^{t} &\Phi\left(\int_{-\infty}^0 h\left(s{+}\sum_{i=1}^{n-1} X^m_i{-}x\right)m(\diff x){+}\sum_{k{=}0}^{n-1} h\left(s{+}\sum_{i=n-k}^{n-1} X^m_i\right)\right)\,\diff s\\
&\le     \int_0^{t}\Phi_{\beta_0}\left(\int_{-\infty}^0 h\left(s{+}\sum_{i=1}^{n-1} X^m_i{-}x\right)m(\diff x){+}\sum_{k{=}0}^{n-1} h\left(s{+}\sum_{i=n-k}^{n-1} X^m_i\right)\right)\,\diff s\\
&\le     \int_0^{t}\Phi_{\beta_0}\left(\int_{-\infty}^0 h\left(s{+}\sum_{i=1}^{n-1} Y^m_i{-}x\right)m(\diff x){+}\sum_{k{=}0}^{n-1} h\left(s{+}\sum_{i=n-k}^{n-1} Y^m_i\right)\right)\,\diff s,
\end{align*}
by the monotonicity properties of $h$ and $\Phi_{\beta_0}$. Hence we deduce that  the relation $Y^m_n{\le}X^m_n$ holds.
\end{proof}

\subsection{Existence Theorem}
We can now formulate our main existence result. 
\begin{theorem}\label{ExistenceTheo}
If the activation function function $\Phi$ is continuous with  $\Phi(0){>}0 $ and
the quantities
  \[
\alpha =\int_0^{+\infty} h(t) \diff t \text{ and }        \beta = \limsup_{y\to+\infty}\frac{\Phi(y)}{y}
\]
are such that $\alpha\beta{<}1$, then there exists a stationary Hawkes  point process. 
\end{theorem}
\begin{proof}
The assumptions imply that there exist $\nu{>}0$ and $\beta_0{\ge}0$ such that  $\alpha\beta_0{<}1$ and $\Phi(x){\le}\Phi_{\beta_0}(x)$ for all $x{\ge}0$, where $\Phi_{\beta_0}$ is defined by Relation~\eqref{phib}.

The Markov chain $({\cal X}_n^0)$, resp. $({\cal Y}_n^0)$, associated to $\Phi$, resp. to $\Phi_{\beta_0}$,  with initial point $({+}\infty)$, is denoted by
\[
{\cal X}_n^0{=}(X_n^0,X_{n-1}^0,\ldots,X_1^0,{+}\infty) \text{ and }
{\cal Y}_n^0{=}(Y_n^0,Y_{n-1}^0,\ldots,Y_1^0,{+}\infty), 
\]
see Definition~\ref{MarkDef}.

We proceed  in the spirit of~\citet{Kryloff_Theorie_1937}, let $\overline{\R}_+{=}\R_+{\cup}\{{+}\infty\}$ and $Q_n$ the distribution on $\overline{\R}_+^{\N}$ given  by
\[
Q_n(F)=\frac{1}{n}\sum_{i=1}^n \E\left(F(X_i^0,X_{i-1}^0,\ldots,X_{i-k+1}^0)\right),
\]
if $F$ is a Borelian function on  $\overline{\R}_+^{\N}$  depending only on the first $k$ coordinates, with the convention that $X^0_i{=}{+}\infty$ for $i{\le}0$. The space $\overline{\R}_+^{\N}$  is endowed with the topology defined by the distance~\eqref{DistS}. 

The sequence of probability distributions $(Q_n)$ on the compact space $\overline{\R}_+^{\N}$ is clearly tight. Let $(Q_{n_p})$ be convergent sequence whose limit is $Q_\infty$. We now prove that $Q_\infty$ is an invariant distribution for the Markov chain $({\cal X}_n^0)$.  Due to definition of $(Q_n)$, it is clear that the probability $Q_\infty$ on $\overline{\R}_+^\N$ is invariant by the shift operator $(x_i){\to}(x_{i+1})$.

Let ${\cal Z}_p{=}(Z_1^p,Z_2^p,\ldots)$ be a random variable on $\overline{\R}_+^{\N}$ with distribution $Q_{n_p}$  and $E_0$ an independent exponential random variable with parameter $1$, ${\cal Z}_\infty{=}(Z_1^\infty,Z_2^\infty,\ldots)$ is  a random variable with distribution $Q_\infty$.

If $F_1$ is an exponential random variable independent of ${\cal Z}_p$,  by definition of $(Q_{n_p})$ we have, for the convergence in distribution,
\[
\lim_{p\to+\infty} ({\cal T}({\cal Z}_p,F_1),Z_1^p,Z_2^p,\ldots)= {\cal Z}_\infty,
\]
If we  prove  that, for the convergence in distribution,
\begin{equation}\label{eqCVEqui}
\lim_{p\to+\infty} {\cal T}({\cal Z}_p,E_0){=}{\cal T}({\cal Z}_\infty,E_0),
\end{equation}
the probability distribution $Q_\infty$ will be an invariant distribution for the Markov chain $({\cal X}_n)$ and therefore that there exists a stationary Hawkes process by Proposition~\ref{InterPalmTheo}.

We fix $T{>}0$, $p{\ge}1$, we have, by definition of $Q_{n_p}$
\begin{equation}\label{TZeq}
\P({\cal T}({\cal Z}_p,E_0){\ge}T)=
\frac{1}{n_p}\sum_{i=1}^{n_p}\E\left(\exp\left({-}\int_0^{T}
\Phi\left(\sum_{k{=}0}^{i} h\left(s{+}\hspace{-3mm}\sum_{j=i+1-k}^{i}\hspace{-2mm} X^0_j\right)\right)\diff s\right)\right). 
\end{equation}
Our proof is carried out in four steps.

\bigskip
\noindent{\sc Step 0}: an integrability property of ${\cal Z}_\infty$.\\
We use the notations of Section~\ref{MonSec}.  With Proposition~\ref{propmono}, we can construct a coupling such that $Y^0_n{\le}X^0_n$ holds almost surely for all $n{\ge}K{>}0$.
\[
\sum_{k=0}^{K}h\left(\sum_{j=i+1-k}^i X^0_{j}\right)\le \sum_{k=0}^{i}h\left(\sum_{j=i+1-k}^i Y^0_{j}\right),
\]
we get,
\[
\frac{1}{n_p}\sum_{i=K}^{n_p} \sum_{k=0}^{K}h\left(\sum_{j=i+1-k}^i X^0_{j}\right)\le
\frac{1}{n_p}\sum_{i=K}^{n_p}\sum_{k=0}^{i}h\left(\sum_{j=i+1-k}^i Y^0_{j}\right),
\]
and, by taking the expected value, this gives the relation
\[
\frac{1}{n_p}\sum_{i=K}^{n_p} \E\left(\sum_{k=0}^{K}h\left(\sum_{j=i+1-k}^i X^0_{j}\right)\right)\le
\frac{1}{n_p}\sum_{i=K}^{n_p}\E\left(\sum_{k=0}^{i}h\left(\sum_{j=i+1-k}^i Y^0_{j}\right)\right). 
\]
As $p$ goes to infinity,  the right-hand side is converging to $E(I)$, where $I$ is the random variable  defined in Proposition~\ref{CrumpProp}. The left-hand side can be expressed as
\[
\E\left(\sum_{k=0}^{K}h\left(\sum_{j=1}^k Z^p_{j}\right)\right),
\]
and this term is converging to
\[
\E\left(\sum_{k=0}^{K}h\left(\sum_{j=1}^k Z^\infty_{j}\right)\right). 
\]
We have thus obtained that, 
\[
\E\left(\sum_{k=0}^{K}h\left(\sum_{j=1}^k Z^\infty_{j}\right)\right)\le \E(I)<{+}\infty,
\]
by letting $K$ go to infinity, this gives the relation
\begin{equation}\label{eqr1}
\E\left(\sum_{k=0}^{+\infty}h\left(\sum_{j=1}^k Z^\infty_{j}\right)\right)<{+}\infty,
\end{equation}

\bigskip
\noindent{\sc Step 1}: A truncation argument.\\
For $i{>}0$, by using the monotonicity property of $h$,  we obtain the relations, for $s{\ge}0$,
\[
\left\{\sum_{k{=}0}^{i} h\left(s{+}\hspace{-3mm}\sum_{j=i+1-k}^{i}\hspace{-2mm} X^0_j\right) {\ge} C\right\}
\subset
\left\{\sum_{k{=}0}^{i} h\left(\sum_{j=i+1-k}^{i}\hspace{-2mm} X^0_j\right) {\ge} C\right\}
\]
hence, with  the couplings of Propositions~\ref{CrumpProp} and~\ref{propmono} and the notations of its proof, and Relation~\eqref{eqavx2}, we obtain
\begin{multline*}
  \P\left(\inf\left(\sum_{k{=}0}^{i} h\left(s{+}\hspace{-3mm}\sum_{j=i+1-k}^{i}\hspace{0mm} X^0_j\right):s{\ge}0\right) {\ge} C\right)
  \le \P\left(\sum_{k{=}0}^{i} h\left(\hspace{0mm}\sum_{j=i+1-k}^{i}\hspace{0mm} X^0_j\right) {\ge} C\right)\\
\le \P\left(\sum_{k{=}0}^{i} h\left(\sum_{j=i+1-k}^{i}\hspace{-2mm} Y^0_j\right) {\ge} C\right)
= \P\left(\sum_{k{=}0}^{i} h\left(\sum_{j=1}^{k}\widetilde{Y}^i_j\right) {\ge} C\right)\\
\le \P\left(\sum_{k{=}0}^{+\infty} h\left(\sum_{j=1}^{k} \widetilde{Y}^{\infty}_j\right) {\ge} C\right).
\end{multline*}
For $\eps{>}0$, Proposition~\ref{CrumpProp} gives therefore the existence of $C_0{>}0$ such that
\begin{equation}\label{eqst1}
\sup_{i{\ge}1}\P\left(\sup_{s\ge 0}\sum_{k{=}0}^{i} h\left(s{+}\hspace{-3mm}\sum_{j=i+1-k}^{i}\hspace{-2mm} X^0_j\right) {\ge} C_0\right)
\leq \eps.
\end{equation}

For $\delta{>}0$, the function $\Phi$ being uniformly continuous $[0,C_0]$, there exists some $\eta_0$ such that $|\Phi(x){-}\Phi(y)|{\le}\delta$, for all elements $x$ and $y$ of $[0,C_0]$ such that $|x{-}y|{\le}\eta_0$. 

\bigskip
\noindent{\sc Step 2}.\\
For $s{\ge}0$, $i{>}K{\ge}0$, with the same arguments, we get 
\begin{multline*}
\left\{  \left|\Phi\left(\sum_{k{=}0}^{i} h\left(s{+}\hspace{-3mm}\sum_{j=i+1-k}^{i}\hspace{-2mm} X^0_j\right)\right){-}
\Phi\left(\sum_{k{=}0}^{K} h\left(s{+}\hspace{-3mm}\sum_{j=i+1-k}^{i}\hspace{-2mm} X^0_j\right)\right)\right|{>}\delta\right\}\\
\subset
\left\{
\sup_{s\ge 0}\sum_{k{=}0}^{i} h\left(s{+}\hspace{-3mm}\sum_{j=i+1-k}^{i}\hspace{-2mm} X^0_j\right) {\ge} C_0
\right\}
\bigcup
\left\{
\sum_{k{=}K+1}^{i} h\left(s{+}\hspace{-3mm}\sum_{j=i+1-k}^{i}\hspace{-2mm} X^0_j\right){>}\eta_0
\right\}\\
\subset
\left\{
\sup_{s\ge 0}\sum_{k{=}0}^{i} h\left(\sum_{j=i+1-k}^{i}\hspace{-2mm} X^0_j\right) {\ge} C_0
\right\}
\bigcup
\left\{
\sum_{k{=}K+1}^{i} h\left(\sum_{j=i+1-k}^{i}\hspace{-2mm} X^0_j\right){>}\eta_0
\right\},
\end{multline*}
and,  as before, 
\begin{multline*}
\P\left(
\sum_{k{=}K+1}^{i} h\left(\sum_{j=i+1-k}^{i}\hspace{-2mm} X^0_j\right){>}\eta_0
\right)
\le 
\P\left(
\sum_{k{=}K+1}^{i} h\left(\sum_{j=i+1-k}^{i}\hspace{-2mm} Y^0_j\right){>}\eta_0
\right)\\
=
\P\left(
\sum_{k{=}K+1}^i h\left(\sum_{j=1}^{k} \widetilde{Y}^{i}_{j}\right){>}\eta_0.
\right)
\le 
\P\left(
\sum_{k{=}K+1}^i h\left(\sum_{j=1}^{k} \widetilde{Y}^{\infty}_{j}\right){>}\eta_0.
\right)\\
\le 
\P\left(
\sum_{k{=}K+1}^{+\infty} h\left(\sum_{j=1}^{k} \widetilde{Y}^{\infty}_{j}\right){>}\eta_0.
\right). 
\end{multline*}
The random variable $I$ being almost surely finite, Proposition~\ref{CrumpProp}, gives that  there exists $K_0$ such that
\[
\P\left(\sum_{k{=}K_0+1}^{+\infty} h\left(\sum_{j=1}^{k} \widetilde{Y}^{\infty}_{j}\right){>}\eta_0\right)\le \eps.
\]
By gathering these results, with Relation~\eqref{eqst1}, we obtain
\[
\sup_{i{\ge}K_0}\P\left( \sup_{s{\ge}0} \left|
\Phi\left(\sum_{k{=}0}^{i} h\left(s{+}\hspace{-3mm}\sum_{j=i+1-k}^{i}\hspace{-2mm} X^0_j\right)\right){-}
\Phi\left(\sum_{k{=}0}^{K_0} h\left(s{+}\hspace{-3mm}\sum_{j=i+1-k}^{i}\hspace{-2mm} X^0_j\right)\right)\right|{>}\delta\right)\\
\le 2 \eps,
\]
hence
\begin{multline}\label{eqst2}
  \lim_{K\to+\infty} \sup_{i{\ge}K}
\left|
\E\left(\exp\left({-}\int_0^{T}\Phi\left(\sum_{k{=}0}^{i} h\left(s{+}\sum_{j=i+1-k}^{i} X^0_j\right)\right)\diff s\right)\right)
\right.\\\left. {-}\E\left(\exp\left({-}\int_0^{T}\Phi\left(\sum_{k{=}0}^{K} h\left(s{+}\sum_{j=i+1-k}^{i} X^0_j\right)\right)\diff s\right)\right)
\right|=0.
\end{multline}

\bigskip
\noindent{\sc Step 3}.\\
Relations~\eqref{TZeq}  and~\eqref{eqst2} show that, for $\eps{>}0$, there exists $K_0$ such that the relation
\begin{multline}\label{eqst3}
  \left|
\P({\cal T}({\cal Z}_p,E_0){\ge}T)\right.\\\left.-
\frac{1}{n_p}\sum_{i=K_0}^{n_p}\E\left(\exp\left({-}\int_0^{T}
\Phi\left(\sum_{k{=}0}^{K_0} h\left(s{+}\hspace{-3mm}\sum_{j=i+1-k}^{i}\hspace{-2mm} X^0_j\right)\right)\diff s\right)\right)\right| \le\eps
\end{multline}
holds for all $p{\ge}1$.
Relation~\eqref{eqr1} shows that one can choose a constant $K_0$  sufficiently large so that
\begin{multline*}
\left|\E\left(\exp\left({-}\int_0^{T}\Phi\left(\sum_{k{=}0}^{K_0} h\left(s{+}\sum_{j=-k}^{1}  Z^{\infty}_j\right)\right)\diff s\right)\right)-\right.\\\left.
\E\left(\exp\left({-}\int_0^{T}\Phi\left(\sum_{k{=}0}^{+\infty} h\left(s{+}\sum_{j=1}^{k}  Z^{\infty}_j\right)\right)\diff s\right)\right)\right|\le \eps. 
\end{multline*}

With the convergence of the sequence $(Q_{n_p})$, we have
\begin{multline*}
\lim_{p\to+\infty}
\frac{1}{n_p}\sum_{i=K_0}^{n_p}\E\left(\exp\left({-}\int_0^{T}
\Phi\left(\sum_{k{=}0}^{K_0} h\left(s{+}\hspace{-3mm}\sum_{j=i+1-k}^{i}\hspace{-2mm} X^0_j\right)\right)\diff s\right)\right)\\
= \E\left(\exp\left({-}\int_0^{T}\Phi\left(\sum_{k{=}0}^{K_0} h\left(s{+}\sum_{j=1}^{k}  Z^{\infty}_j\right)\right)\diff s\right)\right).
\end{multline*}
Note that the $K_0$ first terms in the expression of $Q_{n_p}$ can be arbitrarily small as $p$ gets large.
With Relation~\eqref{eqst3}, by letting $\eps$ go to $0$,  we obtain Relation~\eqref{eqCVEqui}. 
The theorem is proved. 
\end{proof}

\subsection{Coupling Properties}
We have not been able to obtain an analogue, sufficiently strong,  result for the uniqueness of a stationary Hawkes point process. Without a contracting scheme available, one of the few possibilities is of using a coupling argument of the trajectories of the Markov chain $({\cal X}_n^z)$.  In this section, we investigate the coupling of the Markov chain $({\cal X}_n^z)$ with the Markov chain starting from the empty state. 

We assume that the activation function $\Phi$ has the Lipschitz property,
\begin{equation}\label{Lip}
|\Phi(x){-}\Phi(y)|\le L_\Phi|x{-}y|, \quad \forall x, y{\ge}0,
\end{equation}
for some constant $L_\Phi{>}0$.
Let $f$ be a non-negative Borelian function on $\R_+$, and $\tau_f$ be the random variable such that the random variable
\[
\int_0^{\tau_f} f(u)\diff u
\]
has an exponential distribution with parameter $1$, i.e. for $t{\ge}0$,

\[
\P(\tau_f{\ge} t)=\exp\left({-}\int_0^t f(u)\,\diff u\right).
\]
\begin{lemma}\label{Lemd1}
If  $f$ and $g$ are non-negative integrable functions on $\R_+$, then there exists a coupling of $\tau_f$ and $\tau_g$, such
that
\[
\P(\tau_f{\ne}\tau_g)\leq \int_{\R_+}|f(u){-}g(u)|\,\diff u.
\]
\end{lemma}
\begin{proof}
Define $h_f{=}(f{-}g)^+$, $h_g{=}(g{-}f)^+$, we can take the random variables $\tau_{h_f}$, $\tau_{h_g}$  and $\tau_{f{\wedge}g}$ independent, then it is easily checked that
\[
T_f{=}\tau_{h_f}{\wedge}\tau_{f{\wedge}g}{\steq{dist}}\tau_{f}
\text{ and  } T_g{=}\tau_{h_g}{\wedge}\tau_{f{\wedge}g}{\steq{dist}}\tau_{g},
\]
and
\[
\P(T_f{=}T_g)\ge \P(\tau_{h_f}{\wedge}\tau_{h_g}\ge\tau_{f{\wedge}g})=
\P(\tau_{h_f{+}h_g}\ge\tau_{f{\wedge}g})=\P(\tau_{|f{-}g|}\ge\tau_{f{\wedge}g}),
\]
hence, we obtain,
\begin{multline*}
\P(T_f{\ne}T_g)\le\E\left(\exp\left(-\int_0^{\tau_{|f{-}g|}}f(u){\wedge} g(u)\diff u\right)\right)
\\=\int_0^{+\infty} (|f(v){-}g(v)|)\exp\left(-\int_0^{v}f(u){\wedge} g(u){+}|f(v){-}g(v)|\diff u\right)\diff v
\\\le \int_0^{+\infty} |f(v){-}g(v)|\diff v. 
\end{multline*}
The lemma is proved. 
\end{proof}
The next proposition gives an upper bound on the probability that the Markov chain $({\cal X}_n^z)$ do not couple right from the beginning with the Markov chain starting from the empty state. 
\begin{proposition}\label{CoupIneq}
For $z{\in}{\cal S}_h$, if $\Phi$ is Lipschitz with constant $L_\Phi$,  there exists a coupling of the Markov chains $({\cal X}_n^0)$ and $({\cal X}_n^z)$ of Definition~\ref{MarkDef}, and  $D_1$, $D_2{\in}\R_+$ such that 
\begin{multline}\label{UpCoup}
\P\left(\exists n{\ge}1, X_{n}^0{\ne}X_{n}^z\right)
\\\le L_\Phi\left( D_1\int_{\R_-}\int_0^{+\infty} h\left(s{-}u\right)\diff s\,m_z(\diff u)+
 D_2\int_{\R_-}\int_0^{+\infty} sh\left(s{-}u\right)\diff s\, m_z(\diff u)\right).
\end{multline}
\end{proposition}
\begin{proof}
Since the random variables $X_1^0$ and $X_1^z$ can be expressed as $\tau_{f_0}$ and $\tau_{f_y}$, with
\[
f_0(s)=\Phi(h(s)) \text{ and } f_y(s)= \Phi\left(h(s){+}\int_{0}^{+\infty} h(s{-}u)m_z(\diff u)\right),
\]
Lemma~\ref{Lemd1} shows that
\begin{multline*}
P(X_1^0{\ne}X_1^z)\leq
\int_0^{+\infty} \left|\Phi\left( h\left(s\right)\right){-}
\Phi\left( \int_{\R_-}h(s{-}u)m_z(\diff u){+} h\left(s\right)\right)\right|\,\diff s\\
\leq L_\Phi\int_0^{+\infty}\int_{\R_-}h(s{-}u)m_z(\diff u)\,\diff s.
\end{multline*}
By induction, we obtain that, for $n{\ge}1$,
\begin{align*}
  P(X_{n+1}^0{\ne}X_{n+1}^z{\mid} X_k^0&{=}X_k^z, \forall 1{\le}k{\le}n )\\
  \leq \int_0^{+\infty}&\left| \Phi\left( \sum_{k{=}1}^{n+1} h\left(s{+}\sum_{i=k}^{n}X^z_i \right)\right)\right.\\
  &{-}\left.\Phi\left( \int_{\R_-}h\left(s{+}\sum_{i=1}^{n}X^0_i{-}u\right)m_z(\diff u){+} \sum_{k{=}1}^{n+1} h\left(s{+}\sum_{i=k}^{n}X^0_i \right)\right)\right|\,\diff s\\
\leq L_\Phi&\int_0^{+\infty} \int_{\R_-}h\left(s{+}\sum_{i=1}^{n}X^0_i{-}u\right)m_z(\diff u)\diff s.
\end{align*}
With Proposition~\ref{TnIneq}, there exists some $\beta_0$ such that $\alpha\beta_0{<}1$ and
\[
\sum_{i=1}^{n}X^0_i\ge S_n^+, \text{ with } S_n\steq{def} \sum_{i=1}^{n}E_i{-}\alpha\beta_0,
\]
hence
\begin{align*}
\P\left(\exists n{\ge}1, X_{n}^0{\ne}X_{n}^z\right)&\leq
L_\Phi\sum_{n=0}^{+\infty}\int_0^{+\infty} \int_{\R_-}\E\left(h\left(s{+}\sum_{i=1}^{n}X^0_i{-}u\right)\right)m_z(\diff u)\diff s\\
&\le L_\Phi\sum_{n=0}^{+\infty}\int_0^{+\infty} \int_{\R_-}\E\left(h\left(s{+}S_n^+{-}u\right)\right)m_z(\diff u)\diff s.
\end{align*}
For $u{>}0$, in the same way as in the proof of Proposition~\ref{CrumpProp}, with Fubini's Theorem, we obtain
\begin{align*}
  \sum_{n=0}^{+\infty}\int_0^{+\infty} \E\left(h\left(s{+}S_n^+{+}u\right)\right)\diff s
  &=\sum_{n=0}^{+\infty}\int_0^{+\infty} \E\left(h\left(s+u\right)\ind{s{\ge}S_n^+}\right)\diff s\\
&=\int_0^{+\infty} h\left(s{+}u\right)\E\left({\cal N}_+([0,s])\right)\diff s,
\end{align*}
where ${\cal N}_+$ is defined by Relation~\eqref{Cdef}, with Inequality~\eqref{Cineq} we obtain the desired estimate. 
\end{proof}

 \begin{corollary}\label{corlc1}
 For $x_0{=}(z_i,i{\ge}1){\in}{\cal S}_h$, if $\Phi$ is Lipschitz with constant $L_\Phi$ and 
\[
\int_{\R_-}\int_0^{+\infty} sh\left(s{-}u\right)\diff s \,m_{x_0}(\diff u) < {+}\infty,
\]
then there exists $K_0{>}0$ such that if $a{\ge}K_0$ and $x_1(a){=}(a,z_2,z_3,\ldots)$, then
\[
\P\left( X_{n}^0{=}X_{n}^{x_1(a)},\forall n{\ge}1\right)>\frac{1}{2}. 
\]
 \end{corollary}
 \begin{proof}
It is enough to note that, by monotonicity of $h$,  the right-hand side of Relation~\eqref{UpCoup} goes to $0$ if $z_1$ is replaced by a quantity $a$ sufficiently large and, in particular less that $1/2$ if $a{\ge}K_0$, for some $K_0{>}0$. 
 \end{proof}

\section{The Case of Exponential Memory}
\label{ExpDecSec}
In this section we assume that the function $h$ associated to the memory of previous jumps is exponentially decreasing.
\[
h(u){=}\exp({-}u/\alpha),
\]
for some $\alpha{>}0$. In this case, the past activity of the Hawkes process can be encoded by a one-dimensional Markov process. One of the early analyses is~\citet{oakes_markovian_1975}. \citet{duarte_stability_2016} considers a more general model for which $h$ is the density of the sum of identically exponential random variables. The approach is of describing the  Hawkes process in terms of a multi-dimensional Markov process,  to encode the past activity of the Hawkes process.

In this section, we give a  existence and uniqueness result of the stationary Hawkes chain with a weaker
condition than the classical relation $\alpha L{<}1$, where $L$ is the Lipschitz constant associated to $\Phi$.
The result is obtained by using the Markov process of Section~\ref{MarkSec}.
At the same time an explicit representation of the distribution of the corresponding Palm measure in terms of the
invariant distribution of a one-dimensional Markov chain is obtained.
We conclude this section with non-Lipschitz activation functions $\Phi$ for which the solution of the Hawkes SDE blows-up in finite time.
A limit result gives a scaling description of how accumulation of the points of Hawkes process occurs in this case.

\begin{proposition} \label{propZ}
Let $x{\in}{\cal S}_h$, $m_x$ defined by Relation~\eqref{mx} and ${\cal N}_{m_x}{=}(T_n)$ of Proposition~\ref{HSDEprop},
then
    \begin{multline*}
        (Z_n)\steq{def} \left(\int_{(-\infty,T_n]}h(T_n{-}s){\cal N}_{m_x}(\diff s)\right)\\
        = \int_{(-\infty,0]}\exp\left({-}(T_n{-}u)/\alpha\right)m_x(\diff u){+}
            \sum_{k{=}1}^{n} \exp\left({-}(T_n{-}T_k)/\alpha\right)
    \end{multline*}
is a Markov chain on $(1,+\infty)$ such that, for $n{\ge}0$,
    \begin{equation}\label{eqMark}
        Z_0{=}\int_{(-\infty,0]} e^{s/\alpha}m_x(\diff s), \text{ and }Z_{n+1}{=}1{+}e^{- X_{n+1}/\alpha}Z_{n},
    \end{equation}
where $X_{n+1}{=}T_{n+1}{-}T_n$ is the unique solution of the equation

\begin{equation}\label{eqX3}
\int_0^{X_{n+1}}\Phi\left(e^{- s/\alpha}Z_{n}\right)\,\diff s =E_{n+1},
\end{equation}
where $(E_n)$ are i.i.d. random variables with an exponential distribution with parameter~$1$.
\end{proposition}

\begin{proof}
\begin{align*}
  Z_{n+1} &=    \int_{(-\infty,T_{n+1}]}\exp({-}(T_{n+1}{-}s)/\alpha){\cal N}_{m_x}(\diff s)\\
    &=    1{+}\exp({-}(T_{n+1}{-}T_n)/\alpha)\int_{(-\infty,T_{n}]}\exp({-}(T_{n}{-}s)/\alpha){\cal N}_{m_x}(\diff s)\\
     &=   1{+}\exp({-}X_{n+1}/\alpha)Z_n,
\end{align*}
and Relation~\eqref{eqX0} gives that
\begin{align*}
  E_{n+1}&=\int_{T_n}^{T_{n+1}} \Phi\left( \int_{\R_-}h(s{-}u)m_x(\diff u){+}
    \sum_{k{=}1}^{n} h\left(s{-}T_k \right)\right)\,\diff s\\
  &=\int_{0}^{T_{n+1}{-}T_n} \Phi\left(e^{-s/\alpha}\left(\int_{\R_-}h(T_n{-}u)m_x(\diff u){+}
    \sum_{k{=}1}^{n} h\left(T_n{-}T_k\right) \right)\right)\,\diff s\\
  &=\int_{0}^{X_{n+1}}\Phi\left(e^{-s/\alpha}Z_n\right)\diff s 
\end{align*}
is an exponentially distributed random variable with parameter $1$ and that the sequence $(E_n)$ is i.i.d. 
\end{proof}

\begin{proposition}[Harris ergodicity]\label{ErgExp}
If $\Phi$ satisfies Condition~A-2, and if $\alpha\beta_e{<}1$, where
\begin{equation}\label{betacond}
  \beta_e\steq{def}\limsup_{u\to+\infty} \int_{u-1}^u\frac{\Phi(s)}{s}\,\diff s
\end{equation}
then  the sequence $(Z_n)$ is a Harris Markov chain on $[1,+\infty)$.
\end{proposition}
For a general introduction on Harris Markov chains, see~\citet{nummelin_general_2004}.
\begin{proof}
The proof will be in two steps.  We will first prove,  via a Lyapunov function, that the set $[1, K]$ is recurrent, for some $K{>}0$. Then, we will show that the subset $[1, 2]$ is also recurrent and {\it small}, see~\cite{nummelin_general_2004} for example. Proposition~5.10 of~\cite{nummelin_general_2004} gives then that the Markov chain $(Z_n)$ is Harris ergodic.  

We first exhibit a Lyapunov function for this Markov chain.
Equation~\eqref{eqX3} can be rewritten as, if $Z_0{=}z_0$,
\begin{equation}\label{eqaux1}
E_{1}=\int_0^{X_{1}}\phi\left(e^{-s/\alpha}z_0\right)\diff s=
\alpha\int_{z_{0}e^{- X_{1}/\alpha}}^{z_0}\frac{\Phi\left(u\right)}{u}\,\diff u=
\alpha \int_{Z_{1}-1}^{z_0}\frac{\Phi\left(u\right)}{u}\,\diff u.
\end{equation}
  Let, for $y{>}1$,
  \[
  F(y)= \int_{1}^{y-1}\frac{\Phi\left(u\right)}{u}\,\diff u.
  \]
  Relation~\eqref{eqX3} give the identity
  \[
  \int_{z_0}^{z_0\exp(-X_1/\alpha)} \frac{\Phi(u)}{u}\diff u=\frac{E_0}{\alpha},
  \]
and consequently,
\[
\E_{z_0}(F(Z_{1}){-}F(z_{0}))=
  \int_{z_0-1}^{z_0}\frac{\Phi\left(u\right)}{u}\,\diff u- \frac{1}{\alpha},
  \]
  hence there exist $\eta{>}0$ and  $K{>}0$ such that if $z_0{\ge}K$ then
  \[
\E_{z_0}(F(Z_{1}))-\E(F(z_{0}))\leq \left(\beta_e{+}\eta{-}\frac{1}{\alpha}\right)<0.
\]
The function $F$ is a Lyapunov function for the Markov chain $(Z_n)$.
The interval $[1,K]$ is a {\em recurrent set} for the Markov chain.
See Theorem~8.6 of~\citet{robert_stochastic_2003}.

By using Relation~\eqref{eqaux1}, we obtain, for $z_0{\in}[1,K]$,

\begin{multline*}
\P_{z_0}(Z_1<2)=\P\left(\frac{E_1}{\alpha}>\int_{1}^{z_0}\frac{\Phi(u)}{u}\diff u\right)
= \exp\left({-}\alpha\int_{1}^{z_0}
\frac{\Phi(u)}{u}\diff u\right)\\
\ge  \exp\left({-}\alpha\int_{1}^{K}
\frac{\Phi(u)}{u}\diff u\right)>0,
\end{multline*}
the interval $[1,2]$ is a recurrent set for the Markov chain $(Z_n)$.

For $0{<}t{\le}z_0$, the relation
\[
\P(Z_1{-}1{\le}t)=\exp\left({-}\alpha\int_t^{z_0}\frac{\Phi\left(u\right)}{u}\diff u\right)
\]
gives that the density of $Z_1{-}1$ is given by, for $z_0{\le}K$,
\[
\alpha\frac{\Phi\left(t\right)}{t}\exp\left({-}\alpha\int_t^{z_0}
\frac{\Phi\left(u\right)}{u}\diff u\right){\ge}
\alpha\frac{\Phi\left(t\right)}{t}\exp\left({-}\alpha\int_t^{K}\frac{\Phi\left(u\right)}{u}\diff u\right).
\]
There is a positive lower bound independent of $z_0{\le}K$. 
We can now use the same argument as in example of Section~4.3.3 page~98 of~\citet{meyn_markov_1993} to prove that
$[1,2]$ is a small set. The proposition is proved.
\end{proof}

\begin{definition}
For $z{>}1$ and $y{>}0$,  we define $G_\Phi(z,y)$  by the relation
\begin{equation}\label{defL}
\int_0^{G_\Phi(z,y)} \Phi\left(e^{-s/\alpha}z\right)\diff s =y. 
\end{equation}
\end{definition}

\begin{theorem}\label{ThExpErg}[Invariant distribution of $({\cal X}_n)$]
If $\alpha\beta_e{<}1$, then for any $x{\in}{\cal S}_h$, the Markov chain $({\cal X}_n)$ of
Definition~\ref{MarkDef} converges in distribution to the law of
  \[
\left(G_\Phi\left(Z_{-n}^*,\alpha\int_{Z_{-n+1}^*-1}^{Z_{-n}^*}\frac{\Phi(u)}{u}\diff u\right),n{\ge}1\right)
  \]
where $G_\Phi$ is defined by Relation~\eqref{defL} and $(Z_n^*,n{\in}\Z)$ is the stationary version of the Harris Markov
chain $(Z_n)$ of Proposition~\ref{propZ}.
\end{theorem}
\begin{proof}
For $n{\ge}1$, with the above notation, by definition $X_{n+1}{=}G_\Phi(Z_n,E_{n+1})$, i.e.
   \[
E_{n+1}= \alpha \int_{Z_{n+1}-1}^{Z_n}\frac{\Phi\left(u\right)}{u}\,\diff u.
\]
For $n{\ge}k{\ge}1$, $(X_n,X_{n-1},\ldots,X_{n-k+1})$ are the $k$-first coordinates of ${\cal X}_n$, they can be
expressed as
\begin{multline*}
  \left(G_\Phi(Z_{n-1},E_{n}),G_\Phi(Z_{n-2},E_{n-1}),\ldots,G_\Phi(Z_{n-k},E_{n-k+1}) \right)\\
=\left(G_\Phi\left(Z_{n-1},\alpha\int_{Z_{n}-1}^{Z_{n-1}}\frac{\Phi(u)}{u}\diff u\right),\ldots
  G_\Phi\left(Z_{n-k-1},\alpha\int_{Z_{n-k}-1}^{Z_{n-k-1}}\frac{\Phi(u)}{u}\diff u\right) \right).
\end{multline*}
The Harris ergodicity of $(Z_n)$ implies that the random variable $(Z_{n},Z_{n-1},\ldots,Z_{n-k})$ is converging in
distribution to  $(Z_{0}^*,Z_{-1}^*,\ldots,Z_{-k}^*)$.
The mapping $(z,y){\mapsto}G_\Phi(z,y)$ is continuous, the continuous
mapping theorem concludes the proof of our result.
\end{proof}

Theorem~\ref{ThExpErg} shows that in the case of exponential memory,  the invariant distribution of $({\cal X}_n)$
can be expressed in terms of a one-dimensional stationary  Markov chain. The following corollary rephrases this result in terms of Hawkes processes. This is a direct application of Proposition~\ref{PalmMarkov}. 
\begin{corollary}
  If $\Phi$ is a continuous function such that $\Phi(0){>}0$ and $\alpha\beta_e{<}1$, where $\beta_e$ is defined by Relation~\eqref{betacond}, then there exists a unique stationary Hawkes process. 
\end{corollary}
Note that the condition $\alpha\beta_e{<}1$ is weaker than the classical conditions of the literature: $\Phi$ Lipschitz with Lipschitz constant $\beta$ such that $\alpha \beta{<}1$.
\subsection*{Transient Hawkes processes}

From now on, we assume a polynomial behavior for $\Phi$ so that, for $x{>}0$, \[\Phi(x){=}(\nu{+}\beta x)^\gamma,\]
where $\nu$, $\beta$ and $\gamma$ are positive real numbers.

Theorem~\ref{ThExpErg} shows that $({\cal X}_n)$ is converging in distribution for all $\alpha$, $\nu$ and $\beta$
when $\gamma{<}1$, and, when $\gamma{=}1$, the convergence occurs if $\alpha\beta{<}1$.

\begin{proposition}\label{TransMark}
  If $\Phi(u){=}(\nu{+}\beta u)^\gamma$ and if $\gamma{>}1$ with $\beta,\,\nu{>}0$,
  then the Markov process $(Z_n)$ is transient.
\end{proposition}
\begin{proof}
From Relation~\eqref{eqaux1}, we have
  \[
  E_{1}= \alpha \int_{Z_{1}-1}^{z_0}\frac{\Phi\left(u\right)}{u}\,\diff u,
  \]
 where $E_1$ is an exponentially distributed random variable with parameter $1$.
 Let, for $u{\ge}0$, $\Phi_p(u){\steq{def}}(\beta u)^\gamma$ and, on the event
 $\{\gamma E_1{<}\alpha \beta^\gamma z_0^\gamma\}$, we define the variable $Z_{1,p}$ such that
\[
E_{1}= \alpha \int_{Z_{1,p}-1}^{z_0}\frac{\Phi_p\left(u\right)}{u}\,\diff u=
\frac{\alpha\beta^\gamma}{\gamma}\left(z_0^\gamma{-}(Z_{1,p}{-}1)^\gamma\right),
\]
then, since $\Phi_p{\le}\Phi$,  we have $ z_0{-}Z_{1,p}{+}1{\ge}z_0{-}Z_1{+}1{\ge}0$
and  $E_2{\steq{def}}{\gamma E_1}/{(\alpha\beta^\gamma)}$, then
 \[
    z_0{-}Z_{1,p}{+}1=z_0\left(1{-}\left(1{-}\frac{E_2}{z_0^\gamma}\right)^{1/\gamma}\right).
  \]
The elementary inequality
    \[
    1{-}(1-h)^{1/\gamma}\le \frac{2}{\gamma}h, \quad 0{\le}h{\le}1/2,
    \]
    gives the relation, for $a{\ge}1$
\begin{align*}
    \E_{z_0}\left((z_0{-}Z_1{+}1)^a\right)&\leq z_0^a\P(2\gamma E_1{\ge}\alpha \beta^\gamma z_0^\gamma)
    +\E\left((z_0{-}Z_{1,p}{+}1)^a\ind{2\gamma E_1{\le}\alpha \beta^\gamma z_0^\gamma}\right)\\
    &\leq z_0^a\exp\left({-}\frac{\alpha \beta^\gamma}{2\gamma} z_0^\gamma\right)
    +\frac{2^a}{\alpha^a\beta^{a\gamma}}\frac{\E\left(E_1^a\right)}{ z_0^{a(\gamma-1)}}.
\end{align*}
Since $\gamma{>}1$, we deduce that
\begin{equation}\label{LyapF}
\lim_{z_0\to+\infty} \E_{z_0}\left(Z_1{-}z_0\right)=1
\text{ and } 
\sup_{z_0{\ge}1} \E_{z_0}\left((Z_1{-}z_0)^2\right)<{+}\infty. 
\end{equation}
Theorem~8.10 of~\citet{robert_stochastic_2003} shows that the Markov chain is transient.
Strictly speaking Theorem~8.10 is for a Markov chain with a countable state space, nevertheless a glance at the proof
of this result shows that it is also valid in our setting.
The proposition is proved.
\end{proof}
Relation~\eqref{eqX3} of Proposition~\ref{propZ} gives that the  Hawkes Point Process $(T_n)$ of
Proposition~\ref{HSDEprop} is such that
\begin{equation}\label{eqd1}
E_{n+1}{\ge}\beta^\gamma Z_n^\gamma\int_0^{T_{n+1}{-}T_n} e^{-s\gamma/\alpha}\diff s=
\frac{\alpha}{\gamma}\beta^\gamma Z_n^\gamma\left(1{-}e^{-\gamma (T_{n+1}-T_n)/\alpha}\right),
\end{equation}
where $(E_n)$ is an i.i.d. sequence of exponential random variables with parameter~$1$.
Under the assumptions of Proposition~\ref{TransMark} the sequence $(Z_n)$ is converging in distribution to infinity and,
with the last relation, the relation
\[
\lim_{n\to{+}\infty} T_{n+1}{-}T_n=0
\]
holds for the convergence in law.
This result suggests that the  points $(T_n)$ are closer and closer asymptotically.
We investigate this aspect in the rest of the section.

We now study  the asymptotic behavior of $(Z_n)$ in the transient case.
We start with a technical lemma.
\begin{lemma}\label{lemexp}
  If $\gamma{\ge}2$ then, for any $\delta{>}0$,
\begin{equation}\label{ExpEst}
  \sup_{z_0{>}1} \E_{z_0}\left(e^{\delta|Z_1{-}z_0|}\right)<{+}\infty.
\end{equation}
\end{lemma}
\begin{proof}
From Relation~\eqref{eqaux1}, we get, for  $z_0{\ge}1$ and $1{\le}t{\le}z_0$,
\begin{multline}\label{eqaa}
\P_{z_0}\left(z_0{-}Z_1{+}1{\ge}t\right)=\P\left(E_1{\ge}\alpha\int_{z_0-t}^{z_0}\frac{\Phi(u)}{u}\right)
\\ \le\exp\left({-}bz_0^\gamma\left(1{-}\left(1{-}\frac{t}{z_0}\right)^\gamma \right)\right),
\end{multline}
with $b{\steq{def}}{\alpha\beta^\gamma}/{\gamma}$.
Note that, by Definition~\eqref{eqMark} of $Z_1$, $z_0{-}Z_1{+}1{\ge}0$, hence, for $\delta{>}0$,
\begin{align*}
  \E_{z_0}\left(e^{\delta|Z_1{-}z_0{-}1|}\right){-}1&=
  \delta\int_0^{z_0} e^{\delta t} \P_{z_0}\left(z_0{-}Z_1{+}1{\ge}t\right)\,\diff t\\
  &\le \delta\int_0^{z_0}
    \exp\left(\delta t{-}bz_0^\gamma\left(1{-}\left(1{-}\frac{t}{z_0}\right)^\gamma \right)\right)\,\diff t\\
  &=\delta z_0\int_0^1 \exp\left(\delta u z_0{-}bz_0^\gamma\left(1{-}\left(1{-}u\right)^\gamma \right)\right)\,\diff u.
\end{align*}
When $z_0$ is sufficiently large, we split the integral into two terms, 
\[
z_0\int_{1-(1-2\delta/(bz_0^{\gamma-1}))^{1/\gamma}}^1
\exp\left(\delta uz_0{-}bz_0^\gamma\left(1{-}\left(1{-}u\right)^\gamma \right)\right)\,\diff u
\leq z_0\exp\left(-\delta z_0\right),
\]
and
\begin{multline*}
z_0\int_0^{1-(1-2\delta/(bz_0^{\gamma-1}))^{1/\gamma}}
\exp\left(\delta uz_0{-}bz_0^\gamma\left(1{-}\left(1{-}u\right)^\gamma \right)\right)\,\diff u
\\\leq z_0\left(1-\left(1-2\frac{\delta}{bz_0^{\gamma-1}}\right)^{1/\gamma}\right)
\exp\left(\delta z_0\left(1-(1-2\delta/(bz_0^{\gamma-1}))^{1/\gamma}\right)\right).
\end{multline*}
The lemma is proved. 
\end{proof}
\begin{proposition}\label{RatInf}
If $\gamma{\geq}2$, then,  almost surely,
  \[
  \lim_{n\to+\infty}\frac{Z_n}{n}=1.
  \]
\end{proposition}
\begin{proof}
This is a consequence of Relations~\eqref{LyapF} and~\eqref{ExpEst}. Theorem~8.11
of~\citet{robert_stochastic_2003} shows that, almost surely,  the relation
\[
\liminf_{n\to+\infty} \frac{Z_n}{n}\ge 1
\]
holds.
Condition~b) of Theorem~8.11 of~\citet{robert_stochastic_2003} follows from Lemma~\ref{lemexp} and Condition~c)  of this proposition is
replaced in this context by the fact that if ${z_0}{<}K_0$, then there exists $n_0{\ge}1$ such that
$\P_{z_0}(Z_{n_0}{\ge}K_0){>}0$.

Notice that $Z_n{\le}n{+}z_0$ holds for all $n{\ge}1$, we get therefore that almost surely
\[
\lim_{n\to+\infty} \frac{Z_n}{n}=1.
\]
The proposition is proved. 
\end{proof}
For $x{\in}{\cal S}_h$ and $m_x$ defined by Relation~\eqref{mx}, if ${\cal N}_{m_x}{=}(T_n)$ the point process of
Proposition~\ref{HSDEprop}, we know that the sequence $(T_{n+1}{-}T_n)$ is converging in distribution to $0$.
The following proposition gives a much more detailed description of the accumulation of points:

For $n{\ge}1$, the point process seen from the $n$th point, i.e. the point process $(T_n{-}T_k, 1{\le}k{\le}n)$
scaled by the factor $n^\gamma$ converges in distribution to a Poisson point process.

\begin{theorem}[Asymptotic behavior of points of a transient Hawkes process] \label{TransientH}
 Assume that $\Phi(u){=}(\nu{+}\beta u)^\gamma$ with $\gamma{\ge}2$ and $\beta$, $\nu{>}0$,
 if  $x{\in}{\cal S}_h$, $m_x$ defined by Relation~\eqref{mx} and ${\cal N}_{m_x}{=}(T_n)$ the point process of
 Proposition~\ref{HSDEprop}, then the point process
 \[
 \left(n^\gamma(T_n{-}T_{k}),1{\le}k{<}n\right)
 \]
 converges in distribution  to a Poisson process with rate $\beta^\gamma$.
\end{theorem}
\begin{proof}
  Define ${\cal P}_n{=}(n^\gamma(T_n{-}T_{k}),1{\le}k{<}n)$.  Relation~\eqref{eqX3} and the representation of $(T_n)$
  in terms of the Markov chain $(Z_n)$ of Proposition~\ref{propZ}  give the identity
\begin{equation}\label{eqbb}
  E_{n+1}= \int_0^{X_{n+1}}\left(\nu{+}\beta Z_ne^{-s/\alpha}\right)^\gamma\diff s,
\end{equation}
where $(E_n)$ is an i.i.d. sequence of exponential random variables with parameter $1$ and, as before,
$X_{n+1}{=}T_{n{+}1}{-}T_n$.

For any $\delta{>}0$ and $n{\ge}1$, Relation~\eqref{eqbb} gives the inequality
\begin{equation}\label{eqc}
\P\left(X_{n+1}{\ge}\delta\right)\le\P\left(\frac{\alpha}{\gamma}\beta^\gamma
\left(1{-}e^{-\delta\gamma/\alpha}\right)Z_{n}^\gamma\le E_{n+1}\right).
\end{equation}

On the event ${\cal E}_{n, \delta}$,
\[
  {\cal E_{n, \delta}}\steq{def} \left\{ X_{n+1}\le \delta, \frac{\nu}{\beta}
  e^{\delta/\alpha}\le Z_n \right\},
\]
Relation~\eqref{eqbb} shows that
\[
\beta^\gamma e^{-\delta\gamma/\alpha} Z_n^\gamma X_{n+1}\le E_{n+1},
\]
and therefore that the sequence of random variables $(Z_n^\gamma X_{n+1})$ is therefore tight.

The elementary relation
\[
(1{+}h)^{\gamma}-1\le C_1h, \quad  0{\le}h{\le}1,
\]
with $C_1{=}\gamma2^{\gamma{-}1}$ and Relation~\eqref{eqbb} give the following inequality 
\begin{multline*}
0\le E_{n+1} 
-  \beta^\gamma Z_n^\gamma\int_0^{X_{n+1}}e^{-s\gamma/\alpha}\diff s\\
\leq C_1\nu\left(\beta Z_n\right)^{\gamma-1} \int_0^{X_{n+1}}e^{-(\gamma-1)s/\alpha}\diff s
\le C_1\nu\beta^{\gamma-1}X_{n+1} Z_n^{\gamma-1},
\end{multline*}
on the event ${\cal E}_{n,\delta}$.
Using the fact that $(Z_n^\gamma X_{n+1})$ is tight and the almost sure convergence of $(Z_n/n)$ to
$1$ of Proposition~\ref{RatInf} shows that the random variables $(X_{n+1} Z_n^{\gamma-1})$ is converging in distribution to
$0$. For a sufficiently large $n$, Relation~\eqref{eqc} shows that the probability of the event ${\cal E}_{n, \delta}$ is arbitrarily close to $1$,  we obtain that the sequence of random variables 
\[
(Y_n)\steq{def} \left(\beta^\gamma Z_n^\gamma\int_0^{X_{n+1}}e^{-s\gamma/\alpha}\diff s\right)
\]
is converging in law to an exponential distribution with parameter $1$.

On the event  ${\cal E}_{n, \delta}$, we have
\[
Y_n\le \beta^\gamma Z_n^\gamma X_{n+1}\le e^{\delta\gamma/\alpha}Y_n,
\]
hence, for $x{\ge}0$,
\[
\P(\beta^\gamma Z_n^\gamma X_{n+1}\ge x)\leq
\P\left(e^{\delta\gamma/\alpha}Y_n\ge x\right){+}\P\left({\cal E}_{n, \delta}^c\right),
\]
hence
\[
\limsup_{n\to+\infty}\P\left(\beta^\gamma Z_n^\gamma X_{n+1}\ge x\right)\leq \exp\left(-x e^{-\delta\gamma/\alpha} \right),
\]
and, by letting $\delta$ go to $0$, we obtain
\[
\limsup_{n\to+\infty}\P\left(\beta^\gamma Z_n^\gamma X_{n+1}\ge x\right)\leq e^{-x}.
\]
An analogue lower bound holds for the $\liminf$, hence 
$(\beta^\gamma Z_n^\gamma X_{n+1})$ is converging in distribution to an exponential distribution with parameter $1$. By using  the almost sure convergence of $(Z_n/n)$ to $1$, the same property holds for $(\beta^\gamma n^\gamma X_{n+1})$.

For $p{\ge}1$, the same argument may be used to prove that the sequence of random variables
$(\beta^\gamma n^\gamma X_{n-k},0{\le}k{\le}p{-}1)$ is converging in distribution to the product of $p$  exponential distributions
with parameter $1$.
The key argument is again the almost sure convergence of $(Z_n/n)$ to $1$.

Final step.
Let $f$ be a continuous function on $\R_+$ with compact support included in $[0,K_0]$ for some $K_0{>}0$.
The relations $Z_n{\le}n{+}z_0$ and~\eqref{eqbb} give the inequality, for all $n{\ge}1$,
\[
E_{n+1}\le  \left(\nu{+}\beta(n{+}z_0)\right)^\gamma X_{n+1}\le C_0(n{+}1)^\gamma X_{n+1},
\]
for some constant $C_0$  independent of $n$.

For  $p{\le}n$, we thus have
\[
\P( n^\gamma(T_n{-}T_{n-p})\le K_0)\leq \P\left(\sum_{k=n-p+1}^n k^\gamma X_k\le K_0\right)
\le \P\left(\frac{1}{C_0}\sum_{k=1}^p E_k\le K_0 \right).
\]
We denote ${\cal N}_{\beta^{\gamma}}$ a Poisson process with rate $\beta^{\gamma}$, with associated exponential
random variables $(E^{\beta^{\gamma}}_n)$.

The last inequality gives that ${\cal P}_n$ is stochastically dominated by
${\cal N}_{C_0{+}\beta^{\gamma}}$, a Poisson process with rate $C_0{+}\beta^{\gamma}$, in the sense that
\begin{equation}\label{eqc4}
\P({\cal P}_n((0,K_0])\ge p)\le \P({\cal N}_{C_0{+}\beta^{\gamma}}((0,K_0])\ge p), \text{ for $p{\le}n$.}
\end{equation}
Clearly ${\cal N}_{\beta^{\gamma}}$ is also  stochastically dominated by ${\cal N}_{C_0{+}\beta^{\gamma}}$.
 \begin{multline*}
  \left|\E\left(\exp\left({-}\int f(u){\cal P}_n(u)\right)\right){-}
  \E\left(\exp\left({-}\int f(u){\cal N}_{\beta^\gamma}(u)\right)\right)\right|\\
\le \left|\E\left(\exp\left({-}\hspace{-3mm}\sum_{k=n-p+1}^n f(n^\gamma(T_n{-}T_{n-k}))\right)\right){-}
 \E\left(\exp\left({-}\hspace{-3mm}\sum_{k=n-p+1}^n f\left(E^{\beta^{\gamma}}_k\right)\right)\right)\right|\\
 +2  \P({\cal N}_{C_0{+}\beta^{\gamma}}((0,K_0]){\ge}p){+}2\P({\cal N}_{\beta^{\gamma}}((0,K_0]){\ge} p),
 \end{multline*}
 with Relation~\eqref{eqc4}.
By using the convergence in distribution of the random variables $(\beta^\gamma n^\gamma X_{n-k},0{\le}k{\le}p{-}1)$, the first term of the right-hand side of this inequality can  be made arbitrarily small for $n$ sufficiently large. We can take an integer $p$ independent of $n$  sufficiently large such that the last two terms can be made arbitrarily small.

We have thus  proved that
   \[
   \lim_{n\to+\infty}\E\left(\exp\left({-}\int_{\R_+} f(u){\cal P}_n(u)\right)\right)=
   \E\left(\exp\left({-}\int_{\R_+} f(u){\cal N}_{\beta^\gamma}(u)\right)\right)
   \]
  holds for all continuous functions with compact support.
We can use Theorem~3.2.6 of~\citet{dawson_measure-valued_1993} to conclude the proof of the proposition.
\end{proof}

\printbibliography 

\appendix

\section{General results on point processes}
\label{AppPP}

\subsection*{Point Processes}

We   recall the notations on point measures and the associated random variables, the point processes used throughout
this paper. See~\citet{neveu_processus_1977}, Chapter~1 and~11 of~\citet{robert_stochastic_2003},
and~\citet{dawson_measure-valued_1993} for a general introduction on random measures.

\subsection*{Stationary point processes}

If $m$  is a simple point measure on $\R$, the points of $m$ are enumerated by an increasing sequence $(t_k(m),k{\in}\Z)$,
numbered so that the relations
\begin{equation}\label{eqt1}
t_{-1}(m)< t_0(m)\le 0 <t_1(m)<{\cdots}
\end{equation}
hold, with the convention that $t_k(m){=}{+}\infty$ if there are less than $k{\ge}1$ points of $m$ in $\R_+$, and similarly
on $\R_-$.
The flow of translation operators $(\theta_t)$ on ${\cal M}_p(\R)$ is defined by,  for $t{\in}\R$ and $m{\in}{\cal M}_p(\R)$,
\begin{equation}\label{thetat}
\int_{\R} f(s)\theta_t(m)(\diff s) \steq{def} \int_{\R} f(s{-}t)m(\diff s),
\end{equation}
for any non-negative Borelian function $f$ on $\R$.  

A distribution on ${\cal M}_p(H)$, an element of the set  ${\cal P}({\cal M}_p(H))$, is defined as
{\em a point process} on $H$.

\begin{definition}[Stationarity]
\label{defiStat}
A point process ${\cal N}$ on $\R$ is {\em stationary} with intensity $\lambda{>}0$,  if the random variable ${\cal N}([0,1])$
is integrable and $\E({\cal N}([0,1])){=}\lambda$, and if  its distribution is invariant by translation,
i.e.\ for $t{\in}\R$, $\theta_t({\cal N}){\steq{dist}}{\cal N}$, where  $\theta_t$ is the translation operator defined by
Relation~\eqref{thetat}.
\end{definition}

\subsection*{Palm space of point processes}
The set ${\cal M}_p^0(\R)$ is a subset of elements $m$ of ${\cal M}_p(R)$ such that $m(0){\ne}0$. 
If $m{\in}{\cal M}_p^0(\R)$ then $m$ can be represented either by the non-decreasing sequence $(t_k, k{\in}\Z)$ of its
points, or by the sequence $$x{=}(x_k){=}(t_{-k}{-}t_{-k-1}, k{\ge}0)$$ of  increments between them, with the convention that $t_0{=}0$.

An operator $\widehat{\theta}$ on ${\cal M}_p^0(\R)$ is defined  by, 
\begin{equation}\label{thetaH}
\widehat{\theta}(m){\steq{def}}\theta_{t_1(m)}(m)\ind{t_1(m){<}{+}\infty},
\end{equation}
for $m{\in}{\cal M}_p^0(\R)$, where $t_1(m)$ and $(\theta_t)$ by defined respectively by Relations~\eqref{eqt1}
and~\eqref{thetat}.
A simple point process of ${\cal M}_p^0(\R)$ can be identified to its sequence of inter-arrival times and it is easily
seen that the relation
\[
\left(\rule{0mm}{4mm}(t_{k+1}{-}t_k)(\widehat{\theta}(m)),k{\in}\Z\right)=
\left(\rule{0mm}{4mm}(t_{k+2}{-}t_{k+1})(m),k{\in}\Z\right),
\]
holds.

The mapping $\widehat{\theta}$  is the classical shift operator on sequences. 
If $m{=}(t_n,n{\in}\Z)$ is in ${\cal M}_p^0(\R)$ and $x{=}(x_n){=}(t_n{-}t_{n-1},n{\in}\Z)$, then $m{=}m_{x}$ and $\widehat{\theta}(m){=}m_{\bar{x}}$, with $\bar{x}{=}(x_{n+1})$  if  $m(\R_-){=}m(\R_+){=}{+}\infty$.

\subsection*{Equivalence between stationary point processes and Palm measure}
We now recall some classical results on  stationary point processes on $\R$. 
A stationary simple point process with intensity $\lambda{>}0$ can be equivalently defined by either by
\begin{enumerate}
\item a distribution $Q$ on ${\cal M}_p(\R)$ which is invariant for the continuous flow of translations $(\theta_t)$;
\item a distribution $\widehat{Q}$ on ${\cal M}_p^0(\R)$ called the {\em Palm measure} of $Q$  which is invariant for
the operator $\widehat{\theta}$.
\end{enumerate}
When a) is given, a distribution  $\widehat{Q}$ on ${\cal M}_p^0(\R)$ is constructed via Mecke's Formula so that
property b) holds. See  Chapter~II of \citet{neveu_processus_1977} or Proposition~11.6
of~\citet{robert_stochastic_2003}.

If b) holds, i.e. if $\widehat{Q}$ is given,  then the construction of $Q$ is done with a fundamental construction of
ergodic theory, it is the {\em special flow} associated to the operator $\widehat{\theta}$ and the function $m{\mapsto} t_1(m)$ on
${\cal M}_p^0$.
See Chapter~11 of \citet{cornfeld_ergodic_1982}, and Chapter~10 of \citet{robert_stochastic_2003}.
The distribution $Q$ is expressed as
\begin{equation}\label{Palm}
\int_{{\cal M}_p(\R)} F(m)\,Q(\diff m) =
\lambda \int_{{\cal M}_p^0(\R)} \int_0^{t_1(m)} F(\theta_s(m))\,\diff s \,\widehat{Q}(\diff m),
\end{equation}
for any non-negative Borelian function on ${\cal M}_p(\R)$. 
The probability distribution $\widehat{Q}$ is simply determined by a distribution of the sequence of inter-arrival times
which is invariant by the shift operator.
The space $({\cal M}_p^0(\R),\widehat{\theta},\widehat{Q})$  can be seen as a probability space whose elements are positive
sequences.
It is sometimes called the Palm space of $Q$.

\section{Hawkes processes: a quick review}
\label{App-Rev}
Hawkes processes have been introduced by Hawkes in 1974 in~\citet{hawkes_spectra_1971} as a class of point processes ${\cal N}$, whose stochastic intensity $(\lambda(t))$  depends on previous jumps, i.e through,
\[
\left(\lambda(t)\right) = \left(\Phi\left(\int_{(-\infty, t)}h(t{-}s){\cal N}(\diff s)\right)\right). 
\]
The first Hawkes processes investigated  were restricted to affine activation function~$\Phi$ of the form,
\[
    \Phi(x) = \nu{+}\beta x,
\]
which have a nice representation in terms of age-dependent branching processes, see ~\citet{lewis_branching_1964, vere-jones_stochastic_1970, daley2003pointprocessesI}.  The condition of  existence and uniqueness of stationary Hawkes process in this case is 
\[
  \beta\int_0^{+\infty}h(s)\diff s < 1,
\]
see~\citet{hawkes_cluster_1974}.

The special case where $\nu{=}0$ was investigated in~\citet{bremaud_hawkes_2001}, where a particular interest was dedicated to the critical Hawkes process,
\[
  \beta\int_0^{+\infty}h(s)\diff s = 1.
\]
The same critical Hawkes process, with general immigration rate~$\nu$ is investigated in~\citet{kirchner_note_2017}.
A more precise study of the  statistics of stationary Hawkes processes is developed in~\citet{jovanovic_cumulants_2015}
using a Poisson cluster process representation. The addition of  an external jump process to the linearized Hawkes process,  in view of applications in neuroscience example, has
been considered in~\citet{boumezoued_population_2016}.

\begin{table}[ht!]
\renewcommand{\arraystretch}{1.5}
\scriptsize
  \begin{center}
  \begin{tabular}{>{\centering}m{2in}>{\centering}m{2.5in}}
        \toprule
         \multicolumn{2}{c}{Branching Processes}
     \tabularnewline
        \toprule
        {\bf References}
        &
        \citet{hawkes_cluster_1974,lewis_branching_1964,vere-jones_stochastic_1970}.
    \tabularnewline
        \midrule
        {\bf Assumptions~1}
        \begin{enumerate}[leftmargin=0.2in]
        \item $\Phi(x) = \nu {+} \beta x$ with $\nu>0$ and $\beta\geq 0$,
        \item $h(x) = \exp({-}x/\alpha)$,
        \end{enumerate}
        &
        \citet{oakes_markovian_1975, errais_affine_2010}.
    \tabularnewline
        {\bf Assumptions~2}
        \begin{enumerate}[leftmargin=0.2in]
        \item $\Phi(x) = \nu {+} \beta x$ with $\nu{>}0$ and $\beta{\geq} 0$,
        \item $h {:} \mathbb{R}_+{\rightarrow}\mathbb{R}_+ $,
        \end{enumerate}
        &
        \citet{hawkes_spectra_1971, hawkes_cluster_1974, bremaud_hawkes_2001, jovanovic_cumulants_2015, boumezoued_population_2016, kirchner_note_2017}.
        \tabularnewline
        \toprule
         \multicolumn{2}{c}{Analytical Methods}
    \tabularnewline
        \toprule
        {\bf References}
        &
        \citet{kerstan_teilprozesse_1964}
    \tabularnewline
        \midrule
        {\bf Assumptions~3}
        \begin{enumerate}[leftmargin=0.2in]
        \item $\Phi : \mathbb{R}{\rightarrow}\mathbb{R}_+$ Lipschitz,
        \item $h : \mathbb{R}_+{\rightarrow}\mathbb{R}$ general,
        \end{enumerate}
        &
        \citet{bremaud_stability_1996, chen_multivariate_2017,karabash_stability_2012}.

    \tabularnewline
        \toprule
         \multicolumn{2}{c}{Renewal Properties/Markov Processes}
    \tabularnewline
        \toprule
        {\bf General References}
        &
        \citet{lindvall_lectures_2002, nummelin_general_2004, hairer_convergence_2010}.
    \tabularnewline
        \midrule
        {\bf Assumptions~3}
        \begin{enumerate}[leftmargin=0.2in]
        \item $\Phi {:} \mathbb{R}{\rightarrow}\mathbb{R}_+$ general,
        \item $h {:} \mathbb{R}_+{\rightarrow}\mathbb{R}$ general,
        \end{enumerate}
        &
        \citet{bremaud_stability_1996,hodara_systemes_2016, raad_age_2018, graham_regenerative_2019, raad_renewal_2019, costa_renewal_2020}.
    \tabularnewline
        {\bf Assumptions~4}
            \begin{enumerate}[leftmargin=0.2in]
        \item $\Phi {:} \mathbb{R}{\rightarrow}\mathbb{R}_+$ general,
        \item $h(x) = \exp({-}x/\alpha)$,
        \end{enumerate}
        &
        \citet{duarte_stability_2016}.
    \tabularnewline
    \bottomrule
    \end{tabular}
  \end{center}
    \caption{Existence of stationary Hawkes processes}
  \label{tab:table2}
\end{table}
Hawkes processes with exponential functions $h(x){=}\exp({-}x/\alpha)$ have attracted a particular interest
because the associated counting process has the Markovian property, see~\citet{oakes_markovian_1975}.

Having an affine activation function~$\Phi$ is very helpful because of the corresponding  branching process representation,
However, when studying auto-inhibiting point processes, non-linear activation functions $\Phi$ are natural candidates to use. In this setting 
 the investigation of sufficient conditions for the existence of stationary versions is more delicate.
Most proofs in this domain are based on the functional relation defining
stationary Hawkes processes that can be expressed as a fixed point equation which can be solved through a Picard scheme,
see~\citet{kerstan_teilprozesse_1964} for one of the pioneering papers on this subject. 
\citet{bremaud_stability_1996} has developed this approach when $\Phi$ is unbounded and supposed to
be $\beta$-Lipschitz, with the following condition,
\[
    \beta \int_0^{+\infty}|h(s)|\diff s < 1.
\]
Thinning techniques have been applied to the case of a bounded Lipschitz function $\Phi$ in the same reference,
where the condition
\[
    \int_0^{+\infty}s|h(s)|\diff s < {+}\infty
\]
is sufficient to prove the existence and uniqueness of the stationary version of the Hawkes process.


In~\citet{bremaud_stability_1996}, renewal theory is used to investigate Hawkes processes, with finite memory. 
This approach has been  extended  in~\citet{graham_regenerative_2019, raad_renewal_2019, costa_renewal_2020}.
It is also possible to limit the influence of intensity rate to the last jump of the Hawkes process as
done in~\citet{hodara_hawkes_2014}.A recent study~\citet{raad_age_2018} has added a refractory effect to prevent explosion in the study of non-linear Hawkes
processes.

To go further, it is interesting to consider  Hawkes processes as Markov process in a general state space, either
using counting processes~\citet{duarte_stability_2016} or Markov theory in the state of \cadlag functions~\citet{karabash_stability_2012}.
Coupling methods~\citet{lindvall_lectures_2002} and general Markov theory~\citet{nummelin_general_2004, hairer_convergence_2010} are natural tools in this setting. 

\end{document}